\documentclass[reqno,11pt]{amsart}
\usepackage{setspace}

\usepackage[T1]{fontenc}
\usepackage{lmodern}

\makeatletter
\usepackage{fullpage}
\usepackage{textcmds}  
\usepackage{amsmath, amssymb, amsfonts, amstext, verbatim, amsthm, mathrsfs, stmaryrd}
\usepackage[all,cmtip,2cell]{xy}
\usepackage{pgf,tikz,pgfplots}
\pgfplotsset{compat=1.15}
\usetikzlibrary{arrows}
\usetikzlibrary{positioning}
\usetikzlibrary{quotes}
\usepackage{tikz-cd}

\usepackage{enumerate}
\usepackage{enumitem}
\usepackage[colorlinks=true,linkcolor=blue,citecolor=blue,urlcolor=blue,citebordercolor={0 0 1},urlbordercolor={0 0 1},linkbordercolor={0 0 1}]{hyperref} 
\usepackage[shortalphabetic]{amsrefs} 
\usepackage[nameinlink]{cleveref}

\usepackage{enotez}
\setenotez{backref=true}

\def\makeCal#1{%
\expandafter\newcommand\csname c#1\endcsname{\mathcal{#1}}}
\def\makeBB#1{%
\expandafter\newcommand\csname b#1\endcsname{\mathbb{#1}}}
\def\makeFrak#1{%
\expandafter\newcommand\csname f#1\endcsname{\mathfrak{#1}}}

\count@=0
\loop
\advance\count@ 1
\edef\y{\@Alph\count@}%
\expandafter\makeCal\y
\expandafter\makeBB\y
\expandafter\makeFrak\y
\ifnum\count@<26
\repeat

\theoremstyle{plain}
\newtheorem{thm}{Theorem}[section]
\newtheorem{cor}[thm]{Corollary}
\newtheorem{lem}[thm]{Lemma}

\newtheorem{prop}[thm]{Proposition}

\theoremstyle{definition}
\newtheorem{rem}[thm]{Remark}
\newtheorem{defn}[thm]{Definition}

\newtheorem{setup}[thm]{Setup}
\newtheorem{notn}[thm]{Notation}

\newtheorem{ex}[thm]{Example}

\newtheorem*{thm*}{Theorem}
\newtheorem*{prop*}{Proposition}

\newtheorem{introthm}{Theorem}

\def\rm{\mathrm}

\DeclareMathOperator{\rk}{rank}

\DeclareMathOperator{\Aut}{Aut}

\DeclareMathOperator{\ch}{ch}
\DeclareMathOperator{\DCoh}{D^b_{coh}}

\DeclareMathOperator{\Coh}{Coh}

\DeclareMathOperator{\Cone}{Cone}

\DeclareMathOperator{\Fun}{Fun}
\DeclareMathOperator{\GL}{GL}

\DeclareMathOperator{\Hom}{Hom}
\newcommand{\id}{\mathrm{id}}

\DeclareMathOperator{\Mod}{-Mod}

\newcommand{\pt}{\mathrm{pt}}

\DeclareMathOperator{\rank}{rank}

\DeclareMathOperator{\RHom}{RHom}

\DeclareMathOperator{\Stab}{Stab}

\DeclareMathOperator{\Pic}{Pic}

\DeclareMathOperator{\Ob}{Ob}
\newcommand{\heart}{\heartsuit}

\DeclareMathOperator{\im}{im}

\DeclareMathOperator{\gl}{gl}

\DeclareMathOperator{\logZ}{logZ}
\DeclareMathOperator{\gr}{gr}

\newcommand{\iprec}{\mathop{\prec}^{\mathrm{i}}}
\newcommand{\ipreceq}{\mathop{\preceq}^{\mathrm{i}}}
\newcommand{\isucc}{\mathop{\succ}^{\mathrm{i}}}

\newcommand{\isim}{\mathop{\sim}^{\mathrm{i}}}

\def\bf{\mathbf}

\def\bb{\mathbb}

\usepackage{pbox}
\usepackage[normalem]{ulem}

\newcommand{\lss}{\mathcal{P}}

\makeatletter

\usepackage{babel}
\begin{document}

\title[Quasi-convergence of stability conditions]{Quasi-convergence of stability conditions}

\author[D. Halpern-Leistner]{Daniel Halpern-Leistner}
\address{Department of Mathematics, Cornell University, Ithaca, NY}
\email{daniel.hl@cornell.edu}

\author[J. Jiang]{Jeffrey Jiang}
\email{jj685@cornell.edu}

\author[A. Robotis]{Antonios-Alexandros Robotis}
\address{Department of Mathematics, Columbia University, New York, NY}
\email{a.robotis@columbia.edu}

\begin{abstract}
    We develop a framework relating semiorthogonal de-compositions of a triangulated category $\cC$ to paths in its space of stability conditions. We prove that when $\cC$ is the homotopy category of a smooth and proper idempotent complete pre-triang\-ulated dg-category, every semi\-orthogonal decomposition whose fac\-tors admit a Bridgeland stability condition can be obtained from our framework.
\end{abstract}

\maketitle

\tableofcontents

\linespread{1.15}\selectfont

\section{Introduction}

Derived categories were originally developed by Grothendieck and Verdier as a technical tool to streamline proofs and calculations in homological algebra \cite{Verdierquotient}. Some years later, it was realized that the bounded derived category of coherent sheaves on a variety $X$, written $\DCoh(X)$, is an interesting and subtle invariant of $X$. Bondal and Orlov's seminal work \cite{BondalOrlovsod}, along with many following developments, have suggested that the birational geometry of $X$ should manifest itself through decompositions of $\DCoh(X)$ into simpler pieces.

A \emph{semiorthogonal decomposition} of a triangulated category $\cC$, written $\cC = \langle \cC_1,\ldots,\cC_n \rangle$, is a totally ordered collection of full trian\-gulated subcategories $\cC_i \subset \cC$ that collectively generate $\cC$, and such that one has $\Hom(E,F)=0$ for $E \in \cC_i$ and $F \in \cC_j$ with $i>j$. This implies the existence of unique and functorial filtrations of any object $F \in \cC$ with associated graded pieces $\gr_i(F) \in \cC_i$ for each $i$ \cite{B-KSerre}. This can be a powerful tool for understanding the category $\cC$. For example, all additive invariants of $\cC$ split; e.g., $K_0(\cC) \cong \bigoplus_{i} K_0(\cC_i)$.  

The archetypal example of a semiorthogonal decomposition comes from \cite{Beilinson1978}, which shows that $\DCoh(\bP^n)$ can be semi\-orthogonally decomposed into the categories generated by $\cO(k)$ for $k=0,\ldots,n$, each of which is equivalent to $\DCoh(\pt)$.

In recent years it has become clear that the structure of semiorthogonal decompositions is more intricate than first expected. The Jordan--H\"{o}lder Property for semiorthogonal decompositions (as posed in, e.g., \cite{Kawamata2008derived}) fails even for fairly tame varieties \cite{BGvBSJH}. Furthermore, contrary to initial expectations, there are now numerous examples where $\DCoh(X)$ contains a \emph{phantom subcategory}; i.e., a component of a semiorthogonal decomposition $\cA \subset \DCoh(X)$ with $K_0(\cA) = 0$ \cites{BGvBKSPhantom,GorchOrlov,KrahPhantom}. In \cite{NMMP}, it is proposed that a potential way to rule out these phenomena is to instead consider \emph{polarizable} semiorthogonal decompositions, i.e., decompositions $\DCoh(X) = \langle \cA_1,\ldots, \cA_n\rangle$ such that each $\cA_i$ admits a Bridgeland stability condition \cites{Br07, Bayer_short}.

The objective of this work is to provide a general mechanism for identifying polarizable semi\-orthogonal decompositions of $\cC$ using the manifold of Bridgeland stability conditions, $\Stab_\Lambda(\cC)$. We introduce \emph{quasi-convergent paths} in $\Stab_\Lambda(\cC)$ (\Cref{D:exit_sequence}) which are paths $\sigma_\bullet:[0,\infty) \to \Stab_\Lambda(\cC)$ satisfying two conditions:
\begin{enumerate}
    \item All nonzero objects of $\cC$ have \emph{limit} Harder--Narasimhan (HN) filtrations with sub\-quotient objects in a class of \emph{limit semistable} objects, $\lss_{\sigma_\bullet} \subset \cC$; and  \vspace{2mm}
    
    \item For any pair of limit semistable objects $E,F$, the difference in the log of their central charges $\logZ_t(F) - \logZ_t(E)$ either converges as $t\to\infty$ or diverges along a well-defined ray $\bR_{>0}\cdot e^{i\theta}\subset \bC$.
\end{enumerate}

\begin{rem}
Throughout, our stability conditions are required to satisfy the support property with respect to a fixed homomorphism $v:K_0(\cC)\twoheadrightarrow \Lambda$ to a free Abelian group of finite rank, $\Lambda$. See \Cref{S:background} for background on (pre)stability conditions.
\end{rem}

Condition (2) allows us to partition the collection of limit semistable objects by saying $E \sim F$ if $\logZ_t(E)-\logZ_t(F)$ converges. We can then define subcategories of $\cC$ generated by the limit semistable objects in a given equivalence class, and in good cases these subcategories are pieces of a semiorthogonal decomposition of $\cC$.

More precisely, we use the asymptotics of $\logZ_t(E)-\logZ_t(F)$ to introduce a total preorder $\preceq$ on $\lss = \lss_{\sigma_\bullet}$ whose associated equivalence relation is $\sim$ (\Cref{D:preordering}). For any $E \in \lss$, we let $\cC_{\preceq E}$ be the full subcategory of $\cC$ consisting of objects with limit HN factors that are $\preceq E$, and likewise for $\cC_{\prec E}$. Our first main result is the following.

\begin{introthm}[= \Cref{P:exit_sequence_filtration} +  \Cref{T:prestabilityonquotient}]
\label{T:firsttheorem}
For a quasi-convergent path $\sigma_\bullet$, the $\cC_{\preceq E}$ are thick triangulated subcategories of $\cC$, and:
\begin{enumerate}
    \item Each category $\cC_{\preceq E}/\cC_{\prec E}$ admits a pre-stability condition $\sigma_E$ such that the semistable objects are precisely the images of those $F \in \lss$ with $F \sim E$, and whose central charge is \[Z_E(F) = \lim_{t\to \infty} Z_t(F)/Z_t(E).\]

    \item $\lss$ can be partitioned by a coarser equivalence relation $\isim$ (see \Cref{D:preordering}) such that the categories $\cC^E$ consisting of objects with limit HN factors $\isim E$ are the factors of a semiorthogonal decomposition $\cC = \langle \cC^E \colon E \in \lss /{\sim^i} \rangle$.
\end{enumerate}
\end{introthm}

For any $E,E'\in \lss$ with $E \sim E'$, one has $\cC_{\preceq E} = \cC_{\preceq E'}$ and $\cC_{\prec E} = \cC_{\prec E'}$ so that $\cC_{\preceq E}/\cC_{\prec E} = \cC_{\preceq E'}/\cC_{\prec E'}$. Consequently, one obtains a collection of subcategories $\{\cC_{\preceq E}\}$ naturally indexed by $\lss/{\sim}$. By a slight abuse of notation, we write $E\in \lss/{\sim}$ for the class of $E$ in $\lss/{\sim}$. It is also useful to note that the filtration $\{\cC_{\preceq E}\}_{E\in \lss/{\sim}}$ can be obtained by first semiorthogonally decomposing $\cC = \langle \cC^E:E\in \lss/{\isim}\rangle$ and then filtering the subcategories $\cC^E$ by the thick triangulated subcategories $\cC^E_{\preceq F} := \cC^E\cap \cC_{\preceq F}$, where $E,F\in \lss$ and $E\isim F$. See \Cref{F:filtrationpicture}.

By contrast, the pre-stability condition $\sigma_E$ on $\cC_{\preceq E}/\cC_{\prec E}$ depends on the choice of $E\in \lss$, rather than its class in $\lss/{\sim}$. However, for $E\sim E'$ as above, there is a unique $\alpha \in \bC$ such that $\alpha \cdot \sigma_{E'} = \sigma_E$. (See \Cref{S:background} for the definition of the $\bC$-action on pre-stability conditions.) 

Note that the existence of $\sigma_E$ guarantees that $\rk K_0(\cC_{\preceq E}/\cC_{\prec E})>0$ so the associated graded subcategories of the filtration of \Cref{T:firsttheorem} are never phantoms.

\begin{figure}
    \centering
\begin{tikzpicture}
    \draw (0,0) [fill = orange] rectangle (1.5,1.5);
    \draw (1.5,0) [fill = orange] rectangle (3,1.5);
    \draw (3.0,0) [fill = orange] rectangle (4.5,1.5);
    \draw(4.5,0)  rectangle (6,1.5);
    \draw (6,0) rectangle (7.5,1.5);
    \draw(7.5,0)  rectangle (9,1.5);
    \draw(4.5,0) [fill = orange] rectangle (6,1);
    \draw(9,0) rectangle (10.5,1.5);

    \draw node at (.9,1.8) {$\cC^{E_1}$};
    \draw node at (2.4,1.8) {$\cdots$};
    \draw node at (3.75,1.8) {$\cC^{E_{j-1}}$};
    \draw node at (8.25,1.8) {$\cdots$};
    \draw node at (9.9,1.8) {$\cC^{E_n}$};
    \draw node at (5.4,1.8) {$\cC^{E_{j}}$};
    \draw node at (6.9, 1.8) {$\cC^{E_{j+1}}$};
    \draw node at (5.4,.5) {$\cC^{E_j}_{\preceq F}$};

\end{tikzpicture}

    \caption{To visualize \Cref{T:firsttheorem}, let $E_1,\ldots, E_n\in \lss$ be given such that $\{E_1,\ldots, E_n\}\to \lss/{\isim}$ is an ordered bijection. Then $\cC = \langle \cC^{E_1},\ldots, \cC^{E_n}\rangle$. Given $F\in \lss$ such that $F \sim^{\mathrm{i}} E_j$, one has $\mathcal{C}_{\preceq F} = \langle \mathcal{C}^{E_1},\ldots, \mathcal{C}^{E_{j-1}}, \mathcal{C}_{\preceq F}^{E_j}\rangle$ by \Cref{L:LePreciExtension}.}
    \label{F:filtrationpicture}
\end{figure}

The pre-stability conditions of \Cref{T:firsttheorem} do not necessarily satisfy the support property (\Cref{D:supportproperty}). To remedy this, we introduce the stronger notion of a \emph{numerical} quasi-convergent path (\Cref{D:numerical}) and a support property for such paths (\Cref{D:supportpropertypath}). We then have:

\begin{introthm}[= \Cref{T:stabilityonquotient}]
\label{T:secondtheorem}
If $\sigma_\bullet$ is a numerical quasi-convergent path in $\Stab_\Lambda(\cC)$, then
\begin{enumerate}
    \item each $Z_E$ of \Cref{T:firsttheorem} factors through the torsion free part of $v(\cC_{\preceq E})/v(\cC_{\prec E})$, denoted $\Lambda_E$; and \vspace{2mm}
    \item $\sigma_\bullet$ satisfies the support property for paths if and only if all $\sigma_E$ satisfy the support property with respect to $\Lambda_E$.
\end{enumerate}
\end{introthm}

In \Cref{S:numerical} we prove that for many categories $\cC$ considered in practice, every quasi-convergent path is numerical. For instance, this holds for stability conditions on $\DCoh(X)$ that are numerical in the usual sense (\Cref{E:numerical}).

We also show a partial converse to \Cref{T:firsttheorem}. Suppose $\cC$ is the homotopy category of a smooth and proper pre-triangulated dg-category, which is the case for many examples of interest (\Cref{Ex:examples}). Then any polarizable semiorthogonal decomposition, $\cC = \langle \cC_1,\ldots, \cC_n\rangle$, can be obtained from a quasi-convergent path. More precisely, given a homomorphism $v_i : K_0(\cC_i) \twoheadrightarrow \Lambda_i$ and a stability condition $\sigma_i \in \Stab_{\Lambda_i}(\cC_i)$ for all $i$, by identifying $K_0(\cC) \cong \bigoplus_i K_0(\cC_i)$ and defining $\Lambda = \bigoplus_i \Lambda_i$, we get a homomorphism $v := \bigoplus v_i : K_0(\cC) \to \Lambda$. We strengthen the gluing construction of \cite{CP10} to prove the following:

\begin{introthm}[= \Cref{T:recoveringsod}]
\label{T:thirdtheorem}
For $\cC = \langle \cC_1,\ldots, \cC_n\rangle$ and $(\sigma_i)_{i=1}^n \in \prod_{i=1}^n \Stab_{\Lambda_i}(\cC_i)$ as above, there exists a numerical quasi-convergent path $\sigma_\bullet$ in $\Stab_\Lambda(\cC)$ such that 
\begin{enumerate} 
    \item applying \Cref{T:firsttheorem} recovers $\cC = \langle \cC_1,\ldots, \cC_n\rangle$: for all $1\le i \le n$, there exists $E\in \lss/{\sim}$ such that $\cC^E = \cC_i$; and \vspace{2mm}
    \item for any $E \in \lss$ with $\cC^E=\cC_i$, $\sigma_E$ is equivalent to $\sigma_i$ with respect to the $\bC$-action on $\Stab_{\Lambda_i}(\cC_i)$.
\end{enumerate}
\end{introthm}

In \Cref{S:curves}, we consider examples of quasi-convergent paths in the case of $\cC = \DCoh(X)$ for $X$ a smooth projective curve.

\begin{figure}[hbt!]
\[
\begin{tikzcd}
    \left\{\parbox{3.5 cm}{\centering quasi-convergent $\sigma_\bullet$} \right\}\arrow[r,dashed,"(\ref{T:firsttheorem})"] & \left\{\parbox{5cm}{\centering filtrations $\{\cC_{\preceq E}\}_{E\in \lss/{\sim}}$ with pre-stability condition $\sigma_E$ on $\cC_{\preceq E}/\cC_{\prec E}$ up to $\bC$-action.}\right\} \\
    \left\{\parbox{4cm}{\centering numerical $\sigma_\bullet$ with the support property}\right\} \arrow[u,hook] \arrow[r,"(\ref{T:secondtheorem})",dashed] & \left\{\parbox{5cm}{ \centering filtrations $\{\cC_{\preceq E}\}_{E\in \lss/{\sim}}$ with $[\sigma_E]\in \Stab_{\Lambda_E}(\cC_{\preceq E}/\cC_{\prec E})/\bC$ for each $E\in \lss/{\sim}$}\right\} \arrow[u,hook]  \\
    \left\{\parbox{4cm}{\centering numerical $\sigma_\bullet$ with the support property such that $\sim\:=\isim$}\right\} \arrow[u,hook] \arrow[r,dashed, "(\ref{T:secondtheorem})",shift left] & \left\{\parbox{4.7 cm}{\centering $\cC = \langle \cC^E:E\in \lss/{\isim}\rangle$ with $[\sigma_E] \in \Stab_{\Lambda_E}(\cC^E)/\bC$ for each $E\in \lss/{\isim}$}\right\} \arrow[u,hook] \arrow[l,shift left, dashed,"(\ref{T:thirdtheorem})"]
\end{tikzcd}
\]
\caption{We schematize the three above theorems. The condition $\sim \:= \isim$ means that the relations are equivalent on $\lss$ so that the filtration $\{\cC_{\preceq E}\}_{E\in \lss/{\sim}}$ of \Cref{T:firsttheorem} is admissible with corresponding semiorthogonal decomposition as in the bottom right of the figure.}

\label{F:schematic}
\end{figure}

\subsection*{Related work and acknowledgements}

The inspiration for our construction came from the ``radar screens'' of \cites{symplectomorphism, mmp}. The idea there was, roughly, to study the Landau--Ginzburg models $(Y,W\colon Y \to \bC)$ that are mirror (in the sense of homological mirror symmetry) to certain toric varieties $X$, and to find semiorthogonal decompositions of $\DCoh(X)$ by studying the asymp\-totics of the critical points of $W$ as the Landau--Ginzburg model $(Y,W)$ degenenerates. The degeneration of the mirror $(Y, W)$, which is a variation of complex structure, was chosen to correspond to the toric minimal model program for $X$, which can be thought of as a variation of (complexified) K\"{a}hler structure on $X$.

Our results are intended to be a purely homological construction that captures the same structure without reference to the mirror of $X$. The variation of complexified K\"{a}hler structure on $X$ is replaced by a path in $\Stab_\Lambda(\DCoh(X))$, and instead of critical values of the Landau--Ginzburg mirror, we study the asymptotics of the central charges of semistable objects. The notion of quasi-convergent path that we introduce, and our main results, are used in \cite{NMMP} to propose a non-commutative version of the minimal model program that can be studied for any projective manifold, without reference to its mirror.

A natural question is whether numerical quasi-convergent paths in $\Stab_\Lambda(\cC)$ are actually con\-vergent in some larger space. The paper \cite{HL_robotis} constructs such a partial compactification of $\Stab_\Lambda(\cC)/\bC$, with boundary points corresponding to semiorthogonal decompositions and stability conditions on the factors, along with ad\-ditional data that remembers some information about the asymptotics of $\logZ_t(E)-\logZ_t(F)$ for pairs of limit semistable objects.

We would like to thank Arend Bayer, Eric Chen, Andres Fernandez Herrero, Kimoi Kemboi, Alex Perry, Alexander Polishchuk, Yukinobu Toda, and Xiaolei Zhao for many helpful conversations on the topics in this paper. We also thank the referee for useful comments and improvements. The third author would also like to thank Maria Teresa Mata Vivas for her love and support during the preparation of this paper. The authors were supported by NSF grants DMS-2052936 and DMS-1945478, and the first author was supported by a Sloan Fellowship, FG-2022-18834.

\subsection*{Notation and Conventions}

Throughout the paper $\cC$ denotes a pre-triangulated dg-category over a field $k$ unless otherwise specified. We write $\rm{Ho}(\cC)$ for the assoc\-iated triang\-ulated category. The results in \S\S 2.2-2.5 can be proven working only with triangulated categories; however, in \S 2.6 and \S3 we will need to work with smooth and proper pre-triangulated dg-categories.

Given a dg-category $\cC$ and a dg-subcategory $\cD$, $\cC/\cD$ denotes the quotient dg-category of \cite{Drinfelddg}. A strictly full subcategory $\cD$ is called \emph{thick} if for any $X,Y\in \cC$, $X\oplus Y \in \cD$ implies $X\in \cD$. Quotients of triangulated categories by thick subcategories were defined by Verdier \cite{Verdierquotient}. However, \cite{Drinfelddg}*{Thm. 1.6.2} gives that $\rm{Ho}(\cC/\cD)\simeq \rm{Ho}(\cC)/\rm{Ho}(\cD)$ as triangulated categories. Hence, the reader who prefers working with triangulated categories can do so without any serious loss of comprehension. 

For subcategories $\{\cD_\alpha\}_{\alpha \in I}$ of $\cC$, $[\cD_\alpha: \alpha \in I]$ denotes the smallest full subcategory containing all of the $\cD_\alpha$ that is closed under extensions. If all of the $\cD_\alpha$ are triangulated, then so is $[\cD_\alpha:\alpha \in I]$. If $I$ is a totally ordered set, then $\langle \cD_\alpha:\alpha \in I \rangle$ means that the categories $\cD_\alpha$ with the ordering from $I$ form a semiorthogonal decomposition of their triangulated closure and refers to that decomposition. We do not make any \emph{a priori} admissibility assumptions on the factors of semiorthogonal decomposition.

Let $\cA$ be an Abelian category. A nonempty full subcategory $\cB \subset \cA$ is called a \emph{Serre subcategory} if for any exact sequence $A\to B \to C$ in $\cA$, $A,C\in \cB$ implies $B\in \cB$ --- see \cite{stacks-project}*{\href{https://stacks.math.columbia.edu/tag/02MN}{Tag 02MN}}.

For $A$ a finitely generated Abelian group, we let $A_{\rm{tf}}$ denote its torsion free part. 

\section{Semiorthogonal decompositions from paths}

\subsection{Preliminaries on stability conditions} \label{S:background}

To fix notation and conventions, we recall the definition of Bridgeland stability conditions. We refer to the objects defined in \cite{Br07} as pre-stability conditions.

\begin{defn}
\cite{Br07} For a triangulated category $\cD$, a \emph{slicing} $\cP$ on $\cD$ is a collection of full additive subcategories $\{\cP(\phi):\phi \in \bR\}$ such that 
\begin{enumerate}
    \item $\cP(\phi)[1] = \cP(\phi+1)$ \vspace{2mm}
    \item for $\phi_1>\phi_2$ and $E_i \in \cP(\phi_i)$ for $i=1,2$, $\Hom_{\cD}(E_1,E_2) = 0$\vspace{2mm}
    \item for any $E\in \cD$, there are maps $0 = E_0 \to E_1\to\cdots \to E_n = E$ with $F_i = \Cone(E_{i-1}\to E_{i}) \in \cP(\phi_i)$ for $1\le i\le n$ and $\phi_1>\cdots>\phi_n$.
\end{enumerate}
The objects of $\cP(\phi)$ are called \emph{semistable} of phase $\phi$ and the collection of maps in (3) is called a \emph{Harder--Narasimhan (HN) filtration} of $E$. The $F_i$ are called the \emph{HN factors} of $E$.
\end{defn}

Given a slicing $\cP$, $X\in \cP$ means that $X$ is nonzero and semistable of some phase $\phi \in \bR$; i.e. $X\in \bigcup_{\phi \in \bR} \cP(\phi)\setminus \{0\}$.

\begin{defn}
    A \emph{pre-stability condition} on a triangulated category $\cD$ is a pair $(Z,\cP)$ where $\cP$ is a slicing and $Z:K_0(\cD) \to \bC$ is a group homomorphism such that for all $\phi \in \bR$ and $E\in \cP(\phi)$, $Z(E) = m(E)\cdot\exp(i\pi \phi)$ with $m(E) \in \bR_{>0}$. $m(E)$ is called the \emph{mass} of $E$.
\end{defn}

To finish the definition of a stability condition, we need to introduce the support property of \cite{KS08}.

\begin{defn}
\label{D:supportproperty}
    Let $\cD$ be a triangulated category and suppose given a surjective homo\-morphism $v:K_0(\cD) \twoheadrightarrow \Lambda$, with $\Lambda$ a free Abelian group of finite rank. A pre-stability condition $(Z,\cP)$ satisfies the \emph{support property} with respect to $v$ if there exists a $C \in \bR_{>0}$ such that 
    \[
    \inf_{E\in \cP(\phi),\phi \in \bR} \frac{\lvert Z(E)\rvert}{\lVert v(E)\rVert} \ge C
    \]
    for some (equivalently any) choice of norm on $\Lambda \otimes \bR$. A pre-stability condition $(Z,\cP)$ satisfying the support property with respect to $v:K_0(\cD) \twoheadrightarrow \Lambda$ is called a \emph{stability condition} and we denote the collection of all such stability conditions on $\cD$ by $\Stab_\Lambda(\cD)$.
\end{defn}

The remarkable fact about $\Stab_\Lambda(\cD)$ is that it has a natural structure of a complex manifold such that the projection map $\pi:\Stab_\Lambda(\cC)\to \Hom(\Lambda,\bC)$ is holomorphic. This property is sometimes called the deformation property and was originally proven by Bridge\-land \cite{Br07} for pre-stability conditions satisfying an additional technical condition. Stability conditions enjoy a stronger deformation property than the original one proven in \cite{Br07}. We refer to \cite{Bayer_short} for details.

We have written $\Stab_\Lambda(\cD)$ above for the space of stability conditions satisfying the support property with respect to a fixed $v:K_0(\cD)\twoheadrightarrow \Lambda$. However, whenever a stability condition is mentioned such a $v:K_0(\cD)\twoheadrightarrow\Lambda$ is implicit. Consequently, we may write $\Stab(\cD)$ instead.

The space of pre-stability conditions on $\cD$ carries a natural continuous action by the universal cover $\GL_2^+(\bR)^{\sim}$ of $\GL_2^+(\bR)$ \cite{Br07}. The action preserves the support property so there is an induced $\GL_2^+(\bR)^{\sim}$-action on $\Stab(\cD)$. The subgroup $\bC^\times\subset \GL_2^+(\bR)$ lifts to a subgroup $\bC\subset \GL_2^+(\bR)^\sim$. The action of $\bC$ is as follows: for any $z\in \bC$ and $(Z,\cP) \in \Stab(\cD)$, $z\cdot (Z,\cP) = (e^z\cdot Z, \cP^z)$, where $\cP^z(\phi) = \cP(\phi - \Im(z)/\pi)$. $\bC$ acts freely on $\Stab(\cD)$ and the quotient space $\Stab(\cD)/\bC$ admits a complex manifold structure such that $\Stab(\cD)\to \Stab(\cD)/\bC$ is a holomorphic principal $\bC$-bundle.

For a pre-stability condition $\sigma = (Z,\cP)$ and a nonzero object $E$, it is standard notation to let $\phi^+_\sigma(E)$ and $\phi^-_\sigma(E)$ denote the largest and smallest phase of an HN factor of $E$, respectively. Likewise, the \emph{mass} is defined as $m_\sigma(E) := \sum_i |Z_\sigma(F_i)|$, where $F_1,\ldots,F_n$ are the HN factors of $E$.

In addition to these standard functions, we introduce the following:

\begin{defn}
The \emph{average phase} of a nonzero object $E$ is 
\[
    \phi_\sigma(E) := \frac{1}{m_\sigma(E)}\sum_{i} \phi_\sigma(F_i) \cdot |Z_\sigma(F_i)|.
\]
where $F_1,\ldots,F_n$ are the HN factors of $E$. We also introduce the function
\[
\ell_\sigma(E) := \log m_\sigma(E) + i\pi \phi_\sigma(E),
\]
for any nonzero $E$, and we let $\ell_\sigma(E/F):=\ell_\sigma(E) - \ell_\sigma(F)$.
\end{defn}

When $E$ is semistable of phase $\phi$, $\phi_\sigma(E) = \phi$. The function $\ell_\sigma$ is meant to approximate the ``logarithm of the central charge'' of $E$. To make this precise we observe:

\begin{lem}
\label{L:logZcomp}
Let $\sigma \in \Stab_\Lambda(\cD)$ be given and $0\le \epsilon<1$. If $E \in \cD$ is nonzero and $\phi_\sigma^+(E) - \phi_\sigma^-(E) \le \epsilon$, then there is a unique complex number
\[
\logZ_\sigma(E) = \log \lvert Z_\sigma(E)\rvert + i\pi \theta
\]
such that $\theta \in [\phi_\sigma^-(E),\phi_\sigma^+(E)]$ and $e^{\logZ_\sigma(E)}=Z_\sigma(E)$. Furthermore,
\[
\lvert \Re(\ell_\sigma(E) - \logZ_\sigma(E)) \rvert \le \frac{(\pi\epsilon)^2}{8} + O(\epsilon^4)\:\:\text{and}\:\: \lvert \Im(\ell_\sigma(E) - \logZ_\sigma(E)) \rvert \le \pi \epsilon.
\]
\end{lem}

\begin{proof}
    Let $E$ have HN factors $F_1,\ldots, F_n$. Since $\phi_\sigma^+(E) - \phi_\sigma^-(E) \le \epsilon$, $Z(F_1),\ldots, Z(F_n)$ all lie in some rotation of $\bH$. Consequently, $Z(E) = \sum_{i=1}^n Z(F_i)$ is nonzero and $\log \lvert Z(E)\rvert$ is defined. $\theta \in [\phi_\sigma^-(E),\phi_\sigma^+(E)]$ allows us to choose a branch cut defining the logarithm $\logZ_\sigma(E)$ with the desired properties. Next, using plane geometry one has 
    \begin{align*}
    \lvert \Re(\ell_\sigma(E) - \logZ_\sigma(E)) \rvert & = \lvert \log m_\sigma(E)/\lvert Z(E)\rvert \rvert \le \lvert \log \cos(\pi\epsilon/2)\rvert \\
    & \le \lvert \cos(\pi\epsilon/2) - 1\rvert \le \frac{(\pi\epsilon)^2}{2^2\cdot 2!} + \frac{(\pi\epsilon)^4}{2^4\cdot 4!} + O(\epsilon^6).
    \end{align*}
    Since $\theta,\phi_\sigma(E) \in [\phi_\sigma^-(E),\phi_\sigma^+(E)]$, one has $\lvert \Im(\ell_\sigma(E) - \logZ_\sigma(E))\rvert = \pi \lvert \phi_\sigma(E) - \theta\rvert \le \pi \epsilon$.
\end{proof}

By \Cref{L:logZcomp}, as $\epsilon \to 0$, $\lvert \ell_\sigma(E) - \logZ_\sigma(E)\rvert \to 0$. In particular, for the limit semistable objects that we consider below $\ell_{\sigma_t}(E)$ and $\logZ_{\sigma_t}(E)$ are essentially equivalent for $t$ sufficiently large. 

\subsection{Quasi-convergent paths}

In this section we consider a continuous map $\sigma_\bullet:[a,\infty)\to \Stab(\cC)$. We write $\sigma_t:= (Z_t,\cP_t)$ for its value at $t\in [a,\infty)$, and for $E \in \cC$ we put $\phi_t^+(E) = \phi_{\sigma_t}^+(E)$ and $\phi_t^-(E) = \phi_{\sigma_t}^-(E)$. 

\begin{defn}[Limit semistable objects]
\label{D:lss}
An object $E\in \cC$ is called \emph{limit semistable} if it is non-zero and $\lim_{t\to\infty} \phi_t^+(E)-\phi_t^-(E)=0$. 
\end{defn}

Note that for a limit semistable object $E$, $\ell_t(E)$ is continuous for all $t\gg 0$ by \Cref{L:logZcomp} --- in fact it is continuous for all $t$ by the results of \cite{deformedmass}. We consider germs of real $C^0$ functions at infinity, i.e., elements of $C^0_\infty(\bR) := \varinjlim C^0((a,\infty),\bR)$.\endnote{For $a<b$, $C^0((a,\infty),\bR) \to C^0((b,\infty),\bR)$ is the restriction map.} Write $f\approx g$ if $\lim_{t\to\infty} f(t)-g(t)=0$. This defines an equivalence relation on $C_\infty^0(\bR)$. Given $f \in C^0_\infty(\bR)$, denote by $\lss_{\sigma_\bullet}(f)\subset \cC$ the full subcategory containing $0$ and all limit semistable objects $E\in \cC$ such that $\phi_t^{\pm}(E)\approx f$. We usually omit the $\sigma_\bullet$ from the notation. Note that $\lss(f)$ depends only on the class of $f$ modulo $\approx$.

\begin{lem}
$\lss(f)$ is an extension closed and thus additive subcategory of $\cC$. Moreover, every limit semistable object belongs to a unique $\lss(f)$.
\end{lem}

\begin{proof}
The first claim is by  $\min\{\phi_t^-(E),\phi_t^-(F)\} \leq \phi_t^-(X) \leq \phi_t^+(X) \leq \max\{\phi_t^+(E),\phi_t^+(F)\}$ for any exact triangle $E\to X\to F$. The second claim is immediate.\endnote{If $E$ is limit semistable, then $f(t)=\phi_t(E)$ defines a function such that $\phi_t^+(E)\approx f(t)\approx \phi_t^-(E)$. If $g(t)$ is another such function, then transitivity of $\approx$ implies $f \approx g$.}
\end{proof}

\begin{defn}[Quasi-convergence] \label{D:exit_sequence}
$\sigma_\bullet$ is called \emph{quasi-convergent}
if:
\begin{enumerate}[label=(\roman*)]
    \item For any $E \in \cC$, there exists a filtration $0=E_0 \to E_1 \to \cdots \to E_n = E$ such that the subquotients $G_i:= \Cone (E_{i-1}\to E_i)$ are limit semistable, and for all $i$
    \[
    \liminf_{t\to \infty} \phi_t(G_{i})-\phi_t(G_{i+1}) > 0.
    \]
    We refer to this as a \emph{limit HN filtration} of $E$.\vspace{2mm}
    
    \item \label{I:semistable_difference} For any pair of limit semistable objects $E$ and $F$, 
    \begin{equation*} 
        \lim_{t \to \infty} \frac{\ell_t(E/F)}{1+\lvert \ell_t(E/F)\rvert}\:\:\:\text{exists.}
    \end{equation*}
\end{enumerate}
\end{defn}

\begin{rem}
    To interpret the latter condition of \Cref{D:exit_sequence}, observe the following: Let $z_t$ denote a path in $\bC$ such that
    \[
        \lim_{t\to\infty} \frac{z_t}{1+\lvert z_t\rvert}
    \]
    exists in $\bb{C}$. Then either $z_t$ converges in $\bC$, or $|z_t|$ diverges to $\infty$ and $\arg(z_t) \in \bR / 2\pi \bZ$ converges.
\end{rem}

\begin{lem}
\label{L:LimitHNofsum}
For any pair of limit semistable objects $E,F$, exactly one of the
following holds:
\begin{enumerate}
    \item $E\oplus F$ is limit semistable.\vspace{2mm}
    \item $E \to E\oplus F \to F \to E[1]$ is a limit HN filtration of $E \oplus F$. \vspace{2mm}
    \item $F \to E \oplus F \to E \to F[1]$ is a limit HN filtration of $E \oplus F$.
\end{enumerate}
\end{lem}

\begin{proof}
Suppose $E\oplus F$ is not limit semistable and has limit HN factors $\{G_i\}_{i=1}^n$. Let
$\{X_{i,t}\}$ and $\{Y_{j,t}\}$ denote the $\sigma_t$-HN factors of $E$ and $F$,
respectively. $\forall t$, the $\sigma_t$-HN filtration of $E\oplus F$ has terms of the form
$X_{i,t}, Y_{j,t}$, or $X_{i,t}\oplus Y_{j,t}$ when $\phi_t(X_{i,t}) = \phi_t(Y_{j,t})$. Fix $c>0$ such that $\forall t\gg 0$, $\min_i\{\phi_t(G_i) - \phi_t(G_{i+1})\} > c$, $\max_i\{\phi^+_t(G_i) - \phi^-_t(G_i)\}<c/4$, and $\max\{\phi_t^+(E) - \phi_t^-(E) , \phi^+_t(F) - \phi_t^-(F)\} < c/4$. 

$\forall t\gg0$, the $\sigma_t$-HN filtration of $E\oplus F$ is the concatenation of the $\sigma_t$-HN filtrations of the $G_i$ for all $i$.\endnote{
\begin{prop}
\label{P:endnoteprop}
Let $\sigma_t$ be a quasi-convergent path, $E \in \cC$, and 
$E_1 \to \cdots \to E_n \to E$ be the limit HN filtration 
of $E$ with limit HN factors $\{G_i\}$. Then 
$\forall t \gg 0$, the $\sigma_t$-HN filtration of  $E$ is the concatenation of
$\sigma_t$-HN filtrations of the $G_i$.
In particular, for all $t \gg 0$, the $\sigma_t$-HN filtration of $E$ is obtained by 
refining its limit HN filtration.
\end{prop}

\begin{proof}
This follows from uniqueness of $\sigma_t$-HN filtrations and the fact that for all $t\gg0$, $\phi_t^-(G_i)>\phi_t^+(G_{i+1})$: By assumption, there is $c > 0$ with 
\[
c < \liminf_{t\to\infty}\phi_t(G_i)-\phi_t(G_{i+1})
\]
By taking $t \gg 0$, we have that both $\phi_t(G_i)-\phi_t(G_{i+1}) > c$ and 
$|\phi^\pm_t(G_i)-\phi_t(G_i)| < c/2$. This then implies that 
$\phi^-_t(G_i) > \phi^+_t(G_i+1)$.
\end{proof}
} In particular, $G_1$ has some $X_{i,t}$, $Y_{j,t}$, or $X_{i,t}\oplus Y_{j,t}$ as a factor. If some $X_{i',t}$ is a $\sigma_t$-factor of $G_1$, then all of $\{X_{i,t}\}$ are and $\{Y_{j,t}\}$ are all $\sigma_t$-HN factors of $G_n$, so that $n\le 2$ and (1) or (2) holds.\endnote{If $X_{i',t}$ is a $\sigma_t$-factor of $G_1$, then $\phi(X_{i',t}) > \phi_t(G_1) - c/4$. However, since $\phi^+_t(E) - \phi_t^-(E) < c/4$, it follows that for all $i$, $\phi(X_{i,t}) \in (\phi_t(G_1)-c/4,\phi_t(G_1) + c/4)$. Consequently, the $X_{i,t}$ are all factors of $G_1$. The case $n=2$ occurs when the $Y_{j,t}$ are (all) factors of some $G_n$ for $n\ne 1$. Otherwise, they are all also factors of $G_1$ and so $E\oplus F$ is limit semistable.} 
If some $X_{i,t}\oplus Y_{j,t}$ is a $\sigma_t$-factor of $G_1$, then all $X_{i,t}$ and $Y_{j,t}$ are so that $n=1$ and (1) holds. (3) holds when some $Y_{j',t}$ is a factor of $G_1$ but no $X_{i',t}$ is.
\end{proof}

\begin{cor}
\label{C:LSScomparable}
If $\sigma_\bullet$ is quasi-convergent then for any limit semistable objects $E,F\in \cC$, exactly one of the following holds:
\begin{enumerate} 
\item $\lim_{t\to \infty} \phi_t(E)-\phi_t(F)=0$, \vspace{2mm}
\item $\liminf_{t\to \infty} \phi_t(E)-\phi_t(F) > 0$, \vspace{2mm}
\item $\liminf_{t\to \infty} \phi_t(F)-\phi_t(E)>0$.
\end{enumerate}
\end{cor}

\begin{proof}
Consider $E\oplus F$. In case (1) of \Cref{L:LimitHNofsum}, $\phi_t(E)\approx \phi_t(F)$. In case (2), we have HN filtration $E\to E\oplus F\to F$, so that by definition $\liminf_{t\to\infty} \phi_t(E)-\phi_t(F)>0$. In case (3), we similarly conclude $\liminf_{t\to\infty} \phi_t(F)-\phi_t(E)>0$. 
\end{proof}

In what follows, we consider a fixed quasi-convergent path $\sigma_\bullet$ in $\Stab(\cC)$. Define a relation $\gtrsim$ on $C^0_\infty(\bR)$ by $f\gtrsim g$ if
\begin{equation} \label{E:ordering}
\liminf_{t\to\infty}f(t)-g(t)\ge 0.    
\end{equation}
Note that $f\lesssim g$ and $f\gtrsim g$ is equivalent to $f\approx g$, so $\gtrsim$ descends to a partial order on $C_\infty^0(\bR)/{\approx}$.\endnote{This is not a total order, since taking $f(t)$ and $g(t)$ such that $f(t)-g(t)=\cos t$ we see $f\not\gtrsim g$ and $f\not\lesssim g$.} We define
\[
\mathfrak{A} := \{[\phi_t(E)]\in C_\infty^0(\bR)/{\approx}: E\:\text{is limit semistable}\}.
\]
Write $f>g$ when the inequality \eqref{E:ordering} is strict.
\begin{cor}\label{C:ordering_A}
For $f,g \in \mathfrak{A}$, $f <g$ is the negation of $f \gtrsim g$, and $\gtrsim$ defines a total order on $\mathfrak{A}$.
\end{cor}
\begin{proof}
This follows from \Cref{C:LSScomparable}.
\end{proof}
So, every object $E\in \cC$ has a limit HN filtration with subquotients $G_i\in \lss(\phi_t(G_i))$ such that $\phi_t(G_1)> \cdots > \phi_t(G_n)$ in $\mathfrak{A}$. It follows from the definition that $\Hom_{\cC}(\lss(f),\lss(g))=0$ whenever $f>g$ in $\mathfrak{A}$. Also, if $E\in \lss(f)$, then $E[1]\in \lss(f+1)$. In particular, the collection $\{\lss(f)\}_{f\in \mathfrak{A}}$ defines a t-stability in the sense of \cite{GKRpaper04}. This implies the following key properties, analogous to the ones for slicings:

\begin{prop}
\label{P:tstab}
Suppose  $\sigma_\bullet$ is quasi-convergent.
\begin{enumerate} 
    \item HN filtrations by limit semistable objects are unique up to unique isomorphism of Postnikov systems; \vspace{2mm}
    \item Given $X\in \cC$ with a filtration by $0=X_0\to X_1\to\cdots \to X_n = X$ such that $\Cone(X_{i-1}\to X_i)=Y_i$ has limit HN filtration with subquotients $(Y_{i,1},\ldots, Y_{i,m_i})$ and $\phi_t^-(Y_i)>\phi_t^+(Y_{i+1})$ for each $i$, the limit HN filtration of $X$ has subquotients 
    \[
        (Y_{1,1},\ldots, Y_{1,m_1},Y_{2,1},\ldots, Y_{2,m_2},\ldots, Y_{n,1},\ldots, Y_{n,m_n}).
    \]
    \item If $F$ and $G$ have limit HN filtrations with subquotients $\{A_i\}$ and $\{B_j\}$, respectively, then the subquotients of the limit HN filtration of $F\oplus G$ are 
    \[
        \{A_i:\phi_t(A_i)\not\approx \phi_t(B_j)\:\forall B_j\}\cup \{B_j:\phi_t(B_j)\not\approx \phi_t(A_k)\:\forall A_k\}
    \]
    \[
        \cup\:\{A_i\oplus B_j:\phi_t(A_i)\approx \phi_t(B_j)\}.
    \]
\end{enumerate}
\end{prop}

\begin{proof}
    The claims are immediate by Thm. 4.1, Prop. 4.3(1), and Prop 4.3(3) of \cite{GKRpaper04}, resp.
\end{proof}

\begin{notn}
\label{N:slicingnotation}
By analogy with the usual notion of slicing from \cite{Br07}, $\lss_{\sigma_\bullet}$ denotes the collection of all nonzero limit semistable objects with respect to $\sigma_\bullet$. That is,
\[
\lss_{\sigma_\bullet} = \bigcup_{f\in \mathfrak{A}}\lss_{\sigma_\bullet}(f)\setminus \{0\}
\]
As before, we omit $\sigma_\bullet$ when it is implicit, writing $\lss$ instead. Given a set $S\subset \mathfrak{A}$, $\lss(S)$ denotes the full subcategory of objects $E$ with $a\lesssim \phi_t^-(E) \lesssim \phi_t^+(E)\lesssim b$ for some $a,b\in S$.
\end{notn}

\begin{rem}
Consider a path $\sigma_\bullet$ in $\Stab(\cC)$.
\begin{enumerate}
    \item In the case where $\sigma_\bullet$ is convergent, it is also quasi-convergent in our sense. In this case, $\mathfrak{A}$ consists of germs of constant functions and is thus identified with a subset of $\bR$. So, $\{\lss(f)\}_{f\in \mathfrak{A}}$ defines a slicing in the sense of \cite{Br07}. \vspace{2mm}
    \item Woolf \cite{Woolfmetric} defines a similar notion of \emph{limiting semistable object} $E$ of phase $\theta$ with respect to a path $\sigma_\bullet$, which requires that $\phi_t^+(E)$ and $\phi_t^-(E)$ converge to some $\theta\in \bR$. This notion is subsumed by ours, since we only require that $\phi_t^+(E) - \phi_t^-(E)\to 0$.
\end{enumerate}
\end{rem}

\subsection{Preorders on \texorpdfstring{$\lss$}{P}}
In what follows, $\sigma_\bullet$ is a fixed quasi-convergent path. From $\sigma_\bullet$, we obtain a preorder $\lss$ by first analyzing the imaginary part of $\ell_t(E/F)$ and then the real part. 

\begin{lem} \label{L: preorder_possibilities}
Let $E,F \in \lss$. Exactly one of the following holds:
\begin{enumerate}
    \item $\lim \limits_{t \to \infty} \phi_t(F) - \phi_t(E) = \pm \infty$; or \vspace{2mm}
    \item $\lim_{t\to\infty} \phi_t(F) - \phi_t(E)$ exists and is an integer; or \vspace{2mm}
    \item there exists an $a\in \bZ$ such that $\limsup_{t\to\infty} \phi_t(F) - \phi_t(E) < a$ and $\liminf_{t\to\infty} \phi_t(F) - \phi_t(E) > a-1$. 
\end{enumerate}
In cases (2) and (3), one has 
\begin{enumerate}
    \item[(a)] $\lim_{t\to\infty} \log m_t(F)/m_t(E) = \pm \infty$; or  \vspace{2mm}
    \item[(b)] $\lim_{t\to\infty} \ell_t(F/E)$ exists in $\bC$.
\end{enumerate}
\end{lem}

\begin{proof}
Suppose (1) does not hold. By \Cref{C:ordering_A}, we may assume $\phi_t(F)\lesssim \phi_t(E)$. There exists a maximal $a\in \bb{Z}$ so that $\phi_t(F)+ a -1\lesssim \phi_t(E) \lesssim \phi_t(F) + a$. Suppose (2) does not hold so $\liminf_{t\to\infty} \phi_t(E) - \phi_t(F) > a-1$ by \Cref{C:LSScomparable}. Similarly, $\limsup_{t\to\infty} \phi_t(F) - \phi_t(E) \le a$. If $\limsup_{t\to\infty} \phi_t(F) - \phi_t(E) = a$, then $\liminf_{t\to\infty} \phi_t(E[-a]) - \phi_t(F) = 0$ and by \Cref{C:LSScomparable} we are in case (2). Finally, if (2) or (3) holds, then (a) or (b) holds by \Cref{D:exit_sequence} \ref{I:semistable_difference}.
\end{proof}

\begin{defn} \label{D:preordering}
For 
$E,F \in \lss$, \Cref{L: preorder_possibilities} allows us to define the following relations: 
\begin{enumerate} 
\item $F \iprec E$ if $\lim_{t \to \infty} \phi_t(E) - \phi_t(F) = \infty$ and $E \ipreceq F$ otherwise;\vspace{2mm}
\item $E \isim F$ if $E \ipreceq F$ and 
$F \ipreceq E$; \vspace{2mm}
\item $F \prec E$ if either: i) $F \iprec E$, or ii) $E \isim F$ and $\lim_{t\to\infty}\log \frac{m_t(E)}{m_t(F)} = \infty$;\vspace{2mm}
\item $E\preceq F$ is the negation of $F\prec E$: $E\ipreceq F$, and if $E\isim F$ then $\lim_{t\to\infty} \log \frac{m_t(E)}{m_t(F)} <\infty$; and \vspace{2mm}
\item $E \sim F$ if $E \preceq F$ and $F \preceq E$; i.e., $\lim_{t\to\infty} \ell_t(E/F)$ exists in $\bC$. 
\end{enumerate}
\end{defn}

\begin{lem}
\label{L:orderproperties}
$\ipreceq$ and $\preceq$ are reflexive and transitive, so $\isim$ and 
$\sim$ are equivalence relations. Furthermore, the partial orders on 
$\lss/{\isim}$ and $\lss/{\sim}$
induced by $\ipreceq$ and $\preceq$, respectively, are total.
\end{lem} 

\begin{proof}
We omit the proof of reflexivity. Suppose $E \ipreceq F$ and $F \ipreceq G$. Then, \Cref{L: preorder_possibilities}
implies that $\phi_t(E) - \phi_t(F)$ and $\phi_t(F) - \phi_t(G)$
either converge to $-\infty$ or are eventually each contained in an open interval of 
length $1$. In either case, $\phi_t(E) - \phi_t(G) = \phi_t(E) - \phi_t(F) + \phi_t(F) - \phi_t(G)$
does not tend to $\infty$, so $E \ipreceq G$. The claims about $\preceq$ are analogous.\endnote{Reflexivity is immediate. Suppose $E \preceq F$ and $F \preceq G$. Since $\preceq$ is defined by 
first checking the relation $\ipreceq$, we may assume that $E \isim F \isim G$.
Then, $\log(m_t(E)/m_t(F))$ and
$\log(m_t(F)/m_t(G))$ either converge to $-\infty$ or a finite value,
which again implies that $\log(m_t(E)/m_t(G))$ does not converge 
to $\infty$.}
$\ipreceq$ and $\preceq$ induce total orders by \Cref{L: preorder_possibilities}.
\end{proof}

\subsection{Filtrations of \texorpdfstring{$\cC$}{C}}

We use the preorders $\ipreceq$ and $\preceq$ coming from $\sigma_\bullet$ to filter $\cC$.  

\begin{defn}
\label{D:preorder_relations_subcategories}
Let $E,F\in \lss$ be given. 
\begin{enumerate}   
    \item In \Cref{N:slicingnotation}, $\lss(\ipreceq E)$ denotes the full subcategory of $\cC$ consisting of objects whose limit HN factors $A$ satisfy $A\ipreceq E$. $\lss(\iprec E)$ and $\lss(\isim E)$ are defined analogously. \vspace{2mm}
    \item Due to its prominent role, we use the notation $\cC^E := \lss(\isim E)$. \vspace{2mm}
    \item $\cC_{\preceq F}$ denotes the full subcategory of $\cC$ whose objects have limit HN factors $A$ with $A\preceq F$. $\cC_{\prec F}$ is defined analogously.\endnote{Although the definitions of $\lss(\ipreceq E)$ and $\cC_{\preceq E}$ are nearly identical, we use the slicing notation for the former, because the relation $\iprec$ only depends on the function $\phi_t(-)$, whereas the category $\cC_{\preceq E}$ depends on the masses as well.}\vspace{2mm}
    \item $\cC^E_{\preceq F}$ denotes the full subcategory of $\cC$ such that $\Ob(\cC^E_{\preceq F}) = \Ob(\cC^E)\cap \Ob(\cC_{\preceq F})$. $\cC^E_{\prec F}$ is defined analogously.
\end{enumerate}
\end{defn}

Given $F,G\in \lss$ such that $F\isim G$, by definition $\cC^F = \cC^G, \lss(\ipreceq F) = \lss(\ipreceq G)$, and $\lss(\iprec F) = \lss(\iprec G)$. So, the natural indexing set for these subcategories is $\lss/{\isim}$. Consequently, by a mild abuse of notation we write $F\in \lss/{\isim}$ for the class of $F$ in $\lss/{\isim}$. 

Similarly, if $E, E'\in \lss$ with $E\sim E'$ it is immediate that $\cC_{\preceq E} = \cC_{\preceq E'}$ and $\cC_{\prec E} = \cC_{\prec E'}$. So, the natural way to index these categories is by $\lss/{\sim}$ and as before, $E\in \lss/{\sim}$ refers to the class of $E$ in $\lss/{\sim}$.

\begin{lem} \label{L:no_homs_right_to_left}
Let $E,F \in \lss$ where $E \iprec F$. For any 
$U \in \cC^E$ and $V \in \cC^F$, $\Hom^k(V,U) = 0$ for all $k \in \bZ$.
\end{lem}

\begin{proof}
Suppose $U,V \in \lss$ with $U \iprec V$. Let $k\in \bb{Z}$ be given and choose $t$ sufficiently large that $\phi_t^-(V)>\phi_t^+(U)+k$. By the properties of the slicing of $\sigma_t$, $\Hom(V,U[k])=0$. Next, for $U \in \cC^E$ and $V \in \cC^F$, any factors $A$ and $B$ of the limit HN filtrations of $U$ and $V$, 
respectively, satisfy $\Hom^k(B,A) = 0$ for all $k$, so
$\Hom^k(V,U) = 0$.
\end{proof}

\begin{prop}
\label{P:ImSOD}
For any $E\in \lss$, one has a semiorthogonal decomposition $\cC = \langle \cC^E : E \in \lss/{\isim} \rangle$, where $\lss/{\isim}$ is totally ordered by $\ipreceq$. In particular, $\cC^E$ is a thick pre-triangulated subcategory of $\cC$.
\end{prop}

\begin{proof}
By \Cref{L:orderproperties}, $\lss/{\isim}$ is totally ordered by $\ipreceq$. If $E\iprec F$ then by \Cref{L:no_homs_right_to_left} one has $\Hom^k(\cC^F,\cC^E) = 0$ for all $k \in \bZ$. Furthermore, by coarsening the limit HN filtration, every $E\in \cC$ has a filtration $0 = E_{0}\to E_1\to \cdots \to E_m = E$ with $G_i := \Cone(E_{i-1}\to E_i)$ such that $G_m\iprec \cdots \iprec G_1$. Thus we have our semiorthogonal decomposition. This implies that each $\cC^E$ is thick and pre-triangulated, because $\cC^E$ is the full subcategory of $\cC$ characterized by $\Hom^k(\cC^F,\cC^E) = 0$ for all $F\isucc E$ and $\Hom^k(\cC^E,\cC^G)$ for all $G\iprec E$.
\end{proof}

\begin{cor}\label{C:Prec I pre-triangulated}
$\lss(\ipreceq E)$ and $\lss(\iprec E)$ are thick pre-triangulated sub\-categories of $\cC$.
\end{cor}

\begin{proof}
We consider $\lss(\ipreceq E)$ since the argument for 
$\lss(\iprec E)$ is nearly identical. By definition, 
$\lss(\ipreceq E) \subseteq \langle \cC^F: F\in \lss/{\isim}, F\ipreceq E\rangle$. We will show that this inclusion is an equality, because then \Cref{P:ImSOD} implies that $\lss(\ipreceq E)$ is a factor in a semiorthogonal decomposition of $\cC$, and hence a thick pre-triangulated subcategory.

Suppose  $G \in \langle\cC^F : F\in \lss/{\isim},F\ipreceq E\rangle$. Write the filtration from the semiorthogonal decomp\-osition as $0 = G_{k+1}\to G_k\to \cdots \to G_1 = G$ with $C_i = \Cone(G_{i+1}\to G_i) \in \cC^{F_i}$ for each $i$ where
$F_1 \iprec \cdots \iprec F_k \ipreceq E$. The HN filtration of $G$ is obtained by concatenating the filtrations of the $C_i$,\endnote{See \Cref{P:endnoteprop}.} so $G\in \lss(\ipreceq E)$ and hence
$\lss(\ipreceq E) =  \langle\cC^F : F\in \lss/{\isim},F\ipreceq E\rangle$.
\end{proof}

\begin{lem}\label{L:homology_in_serre_subcat}
Let $\cA \subset \cC$ be the heart of a bounded t-structure, 
$\{H^i\}_{i\in \bZ}$ the associated cohomology functors, and 
$\cB \subset \cA$ a Serre subcategory. The full subcategory 
$\cD \subset \cC$ consisting of objects $E$ such that $H^i(E) \in \cB$ for all $i$ is a thick 
pre-triangulated subcategory of $\cC$.
\end{lem}

\begin{proof}
Omitted.\endnote{Since each $H^i$ is additive, the cohomology of any direct summand of $E \in \cD$ is a summand of the cohomology of $E$. Since $\cB$ is closed under taking summands, $\cD$ is thick. $\cD$ is closed under shifts. Let $X \to Y \to Z$ be an exact triangle in $\cC$ with $X,Z \in \cD$. The long exact sequence of cohomology gives exact sequences $H^i(X) \to H^i(Y) \to H^i(Z)$, so that $H^i(Y)$ is an extension of a subobject of $H^i(Z)$ by a quotient of $H^i(X)$. Since $H^i(X),H^i(Z) \in \cB$ and $\cB$ is a Serre subcategory, $H^i(Y) \in \cB$ and thus $Y\in \cD$.}
\end{proof}

\begin{lem} \label{L:eventual t-structure}
Fix $E \in \lss$. For $A\in \lss$ with $A\isim E$, define $\phi^E(A)=\limsup_{t\to\infty}\phi_t(A)-\phi_t(E).$ Then
\begin{enumerate} 
\item $\{\lss^E(\phi)\}_{\phi \in \bR}$ with $\cP^E(\phi):=\{A\in \lss :\phi^E(A)=\phi\}$ defines a slicing on $\cC^E$; and \vspace{2mm}
\item for any $a\in \bb{R}$, $\cA^E_a:= \cP^E(a-1,a]$ is the heart of a bounded $t$-structure on $\cC^E$.
\end{enumerate} 
Also, the $\cP^E$-filtration of $F\in \cC^E$ agrees with its limit HN filtration.
\end{lem}

\begin{proof}
The claim (2) is a consequence of (1), so we prove (1). By \Cref{L: preorder_possibilities}, $\phi^E(A)$ exists in $\bR$ for exactly those $A\in \lss$ such that $A\isim E$. Suppose $F,G\in \cP^E(\phi)$ and assume $\phi_t(F)\gtrsim \phi_t(G)$ without loss of generality, by \Cref{C:ordering_A}. By hypothesis, $\phi^E(F) = \phi = \phi^E(G)$. Suppose $\liminf_{t\to\infty} \phi_t(F) - \phi_t(G) = I >0$. Then for $t \gg 0$, we have both $\phi_t(F)-\phi_t(G) \geq I/2$ and $\phi_t(F)-\phi_t(E) \leq \phi + I/4$. This implies $\phi_t(G)-\phi_t(E) \leq \phi - I/4$, which contradicts $\phi^E(G)=\phi$. So, $\liminf_{t\to\infty} \phi_t(F) - \phi_t(G) = 0$ and by \Cref{C:LSScomparable} this implies that $\phi_t(F) \approx \phi_t(G)$, as needed. Therefore, $F\oplus G \in \cP^E(\phi)$ and so $\cP^E(\phi)$ is additive for each $\phi \in \bR$. One verifies that $\cP^E(\phi+1) = \cP^E(\phi)[1]$.

Let $\phi_1>\phi_2$ and let $A_i\in \cP^E(\phi_i)$ be given for $i=1,2$. Set $\epsilon=\frac{\phi_1-\phi_2}{2}$. $\exists t_0$ such that $t\ge t_0$ $\Rightarrow$ $\phi_t^+(A_2) - \phi_t(E) < \epsilon + \phi_2$. Also, $\exists t_1\ge t_0$ such that  $\phi_{t_1}^-(A_1) - \phi_{t_1}^+(E) > \phi_1 - \epsilon$. It follows that $\phi_{t_1}^-(A_1) - \phi_{t_1}^+(A_2) = \phi_1-\phi_2 - 2\epsilon >0$ and so $\Hom(A_1,A_2) = 0$. 

Each $F\in \cC^E$ has a limit HN filtration with subquotients $G_i$ satisfying $\phi_t(G_1)>\cdots>\phi_t(G_n)$. This implies $\phi^E(G_1)\ge \cdots \ge \phi^E(G_n)$. If $\phi^E(G_i)=\phi^E(G_{i+1})$, then by the argument in the first paragraph, $\phi_t(G_i) \approx \phi_t(G_{i+1})$, a contradiction.
\end{proof}

\begin{lem}
\label{L:massproportionality}
Let $a\in \bZ$ and $X\in \cA_a^E$ be given. There exist $C\in (0,1]$ and $t_0 \in \bR$ such that $t\ge t_0$ implies $\lvert Z_t(X)\rvert > C\cdot m_t(X).$
\end{lem}

\begin{proof}
    Because $X\in \cA_a^E$, all of its limit HN factors $\{H_i\}$ are in $\cA_a^E$. For $t \gg 0$, the $\sigma_t$-HN filtration refines the limit HN filtration. It follows that if $H_{\min}$ and $H_{\max}$ are the limit HN factors of minimal and maximal asymptotic phase respectively, then $\phi^+_t(X)=\phi^+_t(H_{\max})$ and $\phi_t^-(X)=\phi_t^-(H_{\min})$ for all $t \gg 0$.

    Now choose a small $\epsilon>0$. Applying \Cref{L: preorder_possibilities} to $H_{\min}$ and $H_{\max}$ implies that for $t \gg 0$, $\phi^+_t(X)=\phi^+_t(H_{\max})< \phi_t(E) + a + \epsilon/2$ and $\phi^-_t(X)=\phi^-_t(H_{\min}) > \phi_t(E)+a-1+\epsilon$. It follows that if $\{F_j\}$ are the $\sigma_t$-HN factors of $X$, then $\{Z_t(F_j)\}$ is contained in an open sector of angular width $\pi-\epsilon/2$ in $\bC$. This implies that $|Z_t(X)| > C \cdot \sum_{j} |Z_t(F_j)| = C \cdot m_t(X)$ for $C = \cos(\frac{\pi}{2}(1-\frac{\epsilon}{2}))$ and all $t \gg 0$.
\end{proof}

\begin{lem}
\label{L:eventual_heart_serre_subcategory}
For all $a \in \bZ$, $\cA^E_a \cap \cC_{\preceq E}$
and $\cA^E_a \cap \cC_{\prec E}$ are Serre subcategories of $\cA^E_a$.
\end{lem}

\begin{proof}
    First, we prove that $\cA_a^E\cap \cC_{\preceq E}$ and $\cA_a^E\cap \cC_{\prec E}$ are closed under subobjects and quotient objects. Note that for $F\in \cC^E$, because all limit HN factors of $F$ are $\isim E$, \Cref{L: preorder_possibilities} implies that $F\in \cC^E_{\preceq E}$ (resp. $\cC^E_{\prec E}$) is equivalent to $\lim_{t\to\infty} m_t(F)/m_t(E) \in [0,\infty)$ (resp. = 0). Suppose given an exact sequence $0\to F'\to F\to F''\to 0$ in $\cA_a^E$. Additivity of $Z_t$ gives $Z_t(F) = Z_t(F') + Z_t(F'') = Z_t(F'\oplus F'')$.  By \Cref{L:massproportionality}, there exist $C>0$ and $t_0 \in \bR$ such that $t\ge t_0$ implies 
    \[
        m_t(F) \ge \lvert Z_t(F)\rvert >  C\cdot m_t(F'\oplus F'') = C(m_t(F') + m_t(F'')).
    \]
    So, if $F\in \cA_a^E\cap \cC_{\preceq E}$ (resp. $\cA_a^E\cap \cC_{\prec E}$) then so are $F'$ and $F''$.

    Consider an exact sequence $F\to H \to G$ in $\cA_a^E$ where $F$ and $G$ are in $\cA_a^E \cap \cC_{\preceq E}$. By the first paragraph, $H$ fits into a short exact sequence $0\to F'\to H \to G'\to 0$ with $F',G'\in \cA_a^E\cap \cC_{\preceq E}$. Because $\cA_a^E$ is the heart of a bounded t-structure, $H\in \cA_a^E$. Applying \Cref{L:massproportionality} again, one has $C\cdot m_t(H)/m_t(E) \le (m_t(F') + m_t(G'))/m_t(E)$ for all $t$ sufficiently large and $C>0$. Hence, $\lim_{t\to\infty} m_t(H)/m_t(E) \in [0,\infty)$ and thus $H\in \cA_a^E\cap\cC_{\preceq E}$. The case of $\cA_a^E\cap \cC_{\prec E}$ is analogous.
\end{proof}

\begin{lem}
\label{L:precimthick}
$\cC^E_{\preceq E}$ and 
$\cC^E_{\prec E}$ are thick pre-triangulated subcategories 
of $\cC^E$.
\end{lem}

\begin{proof}
We consider $\cC^E_{\preceq E}$, $\cC^E_{\prec E}$ being similar. Let $\{H^i\}$ denote the cohomology functors associated to $\cA^E_0$. By \Cref{L:eventual t-structure}, the limit HN factors of $H^i(F)$ are shifts of limit HN factors of $F$,\endnote{By definition, $\cP^E(-1,0] = \cA^E_0$ and \Cref{L:eventual t-structure} states that the filtrations in $\cP^E$ are exactly the limit HN filtrations with appropriate phase labeling. Since these $\cP^E$-filtrations are constructed by refining the t-filtration of $\cA^E_0$, the result follows.} whence $H^i(F) \in \cA^E_0 \cap \cC_{\preceq E}$ for all $i$ if and only if $F\in \cC^E_{\preceq E}$. So, \Cref{L:homology_in_serre_subcat} and 
\Cref{L:eventual_heart_serre_subcategory} imply the result.
\end{proof}

\begin{lem}\label{L:LePreciExtension} 
    $\cC_{\preceq E} = \langle \lss(\iprec E), \cC^E_{\preceq E}\rangle$ and $\cC_{\prec E} = \langle \lss(\iprec E), \cC^E_{\prec E}\rangle$.
\end{lem}

\begin{proof}
    This is an immediate consequence of \Cref{P:ImSOD}.\endnote{$X\in \cC_{\preceq E}$ has limit HN factors $G_j$ such that $G_j\iprec E$ or $G_j\in \cC_{\preceq E}\cap \cC^E$. So, by \Cref{P:ImSOD}, there exists a triangle $F'\to X \to F$ with $F\in \lss(\iprec E)$ and $F'\in \cC_{\preceq E}\cap \cC^E.$ \Cref{P:ImSOD} implies $\Hom(\cC_{\preceq E}\cap \cC^E, \lss(\iprec E)) = 0$. The second claim about $\cC_{\prec E}$ is analogous.}
\end{proof}

\begin{cor}
\label{L:precthick}
$\cC_{\preceq E}$ and $\cC_{\prec E}$ are thick pre-triangulated subcategories of $\cC$. 
\end{cor}

\begin{proof}
    \Cref{L:LePreciExtension} implies that $\cC_{\preceq E} = \langle \lss(\iprec E), \cC_{\preceq E}^E \rangle \subset \lss(\ipreceq E) = \langle \lss(\iprec E), \cC^E \rangle$ is precisely the preimage of the subcategory $\cC^E_{\preceq E}$ under the projection to $\cC^E$. The claim then follows from \Cref{L:precimthick} and the fact that the preimage of a thick subcategory is thick.\endnote{Let $\pi:\cC\to \cD$ denote an exact functor of triangulated categories. Let $\cS\subset \cD$ denoted a (full) thick triangulated subcategory. $\pi^{-1}(\cS)$ denotes its essential preimage. Suppose $X\oplus Y\in \pi^{-1}(\cS)$. Then $\pi(X\oplus Y) \cong \pi(X)\oplus \pi(Y)$ and hence $\pi(X)\in \cS$, whence $X\in \pi^{-1}(\cS)$.}
\end{proof}

\begin{thm}\label{P:exit_sequence_filtration}
The collection $\{\cC_{\preceq E}\}_{E\in \lss/{\sim}}$ defines a filtration of $\cC$ by thick pre-triangulated
subcategories refining the admissible filtration $\{\lss(\ipreceq F)\}_{F \in \lss/{\isim}}$. Also, there is an induced filtration $\{\cC_{\preceq E}^F\}_{E\isim F}$ of each $\cC^F$ by thick pre-triangulated sub\-categories.
\end{thm}

\begin{proof}
    For $E,F\in \lss$, the condition $E\prec F$ implies that $\cC_{\preceq E}\subset \cC_{\prec F}\subsetneq \cC_{\preceq F}$. Therefore, $\{\cC_{\preceq E}\}_{E\in \lss/{\sim}}$ is a filtration by thick pre-triangulated subcategories by \Cref{L:precthick}. Similarly, we get a filtration $\{\lss(\ipreceq F)\}_{F \in \lss/{\isim}}$ which is (left) admissible, corresponding to the semiorthogonal decomposition $\cC = \langle \cC^F: F \in \lss/{\isim}\rangle$ by \Cref{P:ImSOD}.\endnote{Following \cite{B-KSerre}, a full subcategory $i:\cD\hookrightarrow \cC$ is \emph{left (resp. right) ad\-missible} if $i$ admits a left (resp. right) adjoint. Equivalently, $\cC = \langle \cD, {}^\perp\cD\rangle$ (resp. $\cC = \langle \cD^\perp,\cD\rangle$) is a semiorthogonal decomposition. A category is \emph{admissible} if it is left and right admissible. We will suppose our subcategories are left admissible here, however for a smooth and proper category $\cC$ left and right admissibility are equivalent. 

A filtration $\{\cD_i\}_{i=1}^n$ (assumed finite only for simplicity) of $\cC$ is called left ad\-missible if each inclusion $\cD_i\hookrightarrow \cD_{i+1}$ is left admissible. Such a filtration gives rise to a semiorthogonal decomposition of $\cC$ by writing $\cD_2 = \langle \cD_1,{}^\perp\cD_1\rangle$, $\cD_3 = \langle \cD_2,{}^\perp \cD_2\rangle = \langle \langle \cD_1,{}^\perp\cD_1\rangle,{}^\perp \cD_2\rangle$, etc. Conversely, given a semiorthogonal decomposition $\cC = \langle \cC_1,\ldots, \cC_n\rangle$, one obtains a left admissible filtration by putting $\cD_i = \langle \cC_1,\ldots, \cC_i\rangle$ for $1\le i \le n$. One can check that these processes are mutually inverse.
} If $E\isim F$, $\cC_{\preceq E}\subset \lss(\ipreceq F)$, and the refinement claim follows. The claim about the induced filtration on $\cC^F$ is immediate using \Cref{L:precimthick}.
\end{proof}

\subsection{Stability conditions on the subquotients}
This section is dedicated to proving the follow\-ing theorem:

\begin{thm}
\label{T:prestabilityonquotient}
For any $E\in \lss$, there exists a unique pre-stability condition $\sigma_E=(Z_E,\cP_E)$ on $\cC_{\preceq E}/\cC_{\prec E}$ such that
\begin{enumerate}
    \item $\cP_E(\phi)$ consists of the essential image of $\lss(\phi_t(E)+\phi) \subset \cC_{\preceq E}$ in the quotient; and \vspace{2mm}
    \item for any limit semistable $F \sim E$,
    \[
    Z_E(F) = \exp(\lim_{t \to \infty} \ell_t(F/E)) = \lim_{t\to \infty} Z_t(F)/Z_t(E).
    \]
\end{enumerate}
Also, if $E \sim E'$, then under the natural identification $\cC_{\preceq E'}/\cC_{\prec E'} = \cC_{\preceq E}/\cC_{\prec E}$, one has $\sigma_E = (\lim_{t\to\infty} \ell_t(E'/E))\cdot \sigma_{E'}$.
\end{thm}

Consider the diagram:
\[
\begin{tikzcd}
    \cC_{\prec E}^E \arrow[r,hook,"\rm{thick}"]\arrow[d,hook]&\cC_{\preceq E}^E \arrow[r,"\pi"]\arrow[d,hook]&\cC_{\preceq E}^E/\cC_{\prec E}^E \arrow[d,dashed,"\exists!J"]\\
    \cC_{\prec E}\arrow[r,hook,"\rm{thick}"]& \cC_{\preceq E}\arrow[r]&\cC_{\preceq E}/\cC_{\prec E}.
\end{tikzcd}
\]
The arrows labeled ``thick'' are inclusions of thick subcategories by \Cref{L:precimthick} and \Cref{L:precthick}, respectively. The universal property of $\pi$ gives a unique morphism of pre-triangulated dg-categ\-ories $J$ fitting into the above diagram \cite[Thm. 1.6.2]{Drinfelddg}. 

\begin{lem}
\label{L:Jequiv}
The functor $J:\cC_{\preceq E}^E/\cC_{\prec E}^E\to \cC_{\preceq E}/\cC_{\prec E}$ is an equivalence.
\end{lem}

\begin{proof}
By \Cref{L:LePreciExtension}, any $X\in \cC_{\preceq E}$ fits into a triangle $S\to X\to Q$ with $S\in \cC_{\preceq E}^E$ and $Q\in \lss(\iprec E)$. This implies essential surjectivity. Consider $Y\to Z$ with $Y\in \cC_{\preceq E}^E$ and $Z\in \cC_{\prec E}$; if $Y\to Z$ factors through $\cC_{\prec E}^E$, $J$ is fully faithful by \cite[Lem. 4.7.1]{K09}. By \Cref{L:LePreciExtension}, there is a triangle $T\to Z \to Q$ with $Q\in \lss(\iprec E)$ and $T\in \cC_{\prec E}^E$. $\Hom(Y,Q) = \Hom(Y,Q[-1]) = 0$ since $Y\in \cC^E$ and $Q\in \lss(\iprec E)$. So, $\Hom(Y,Z) \cong \Hom(Y,T)$ via the induced map and $Y\to Z$ factors through $T \in \cC_{\prec E}^E$.
\end{proof}

\begin{lem}
\label{L:PhaseConvFactors}
    Suppose given $A,B \in \lss$ such that $A\sim B \sim E$ and $\phi^E(A) > \phi^E(B)$. For any diagram $A\leftarrow A'\to B$ with $\Cone(A'\to A) \in \cC_{\prec E}^E$, there exists $A'' \in \cC_{\preceq E}^E$ and a morphism $f:A''\to A'$ such that 
    \begin{enumerate}   
        \item $\Cone(A''\to A) \in \cC_{\prec E}^E$; and \vspace{2mm}
        \item $\Hom(A'',B) = 0$.
    \end{enumerate}
\end{lem}

\begin{proof}
    Because $A\in \cC^E_{\preceq E}\setminus \cC_{\prec E}^E$, $\Cone(A'\to A) \in \cC_{\prec E}^E$ implies that $A' \in \cC_{\preceq E}^E\setminus \cC_{\prec E}^E$. Write the limit HN filtration of $A'$ as $0 = X_0\to X_1\to \cdots \to X_m = A'$ with factors $\{G_i = \Cone(X_{i-1}\to X_i)\}_{i=1}^m$. For all $t$ sufficiently large, one has $\phi_t(G_1)>\cdots > \phi_t(G_m)$. Let $k$ denote the largest index such that $G_k\sim E$. Such an index exists by $A' \in \cC_{\preceq E}^E\setminus \cC_{\prec E}^E$. Put $A'' := X_k$. The morphism $f:A''\to A'$ is the one from the filtration and consequently $\Cone(f) \in \cC_{\prec E}^E$. 

    (1) is a consequence of the octahedral axiom applied to $A''\xrightarrow{f} A'\to \Cone(f)$, $A'\xrightarrow{g} A\to \Cone(g)$, and $A''\xrightarrow{h} A\to \Cone(h)$ where $h = g\circ f$, noting that $\Cone(f)$ and $\Cone(g) \in \cC_{\prec E}^E$.\endnote{The octahedral axiom applied here gives an exact triangle $\Cone(f) \to \Cone(h)\to \Cone(g)$ implying that $\Cone(h)\in \cC_{\prec E}^E$.} 

    The limit HN filtration of $A''$ is the truncation of that of $A'$ and has limit HN factors $\{G_i\}_{i=1}^k$. Since $A\sim E \sim G_k$, $\phi^E(A) = \lim_{t\to\infty} \phi_t(A) - \phi_t(E)$ and likewise for $\phi^E(G_k).$ In particular, $\lim_{t\to\infty} \phi_t(A) - \phi_t(G_k) \in (a-1,a]$ for some $a\in \bZ$. Consider the heart of a bounded t-structure $\cA = \cA_a^{G_k}$ on $\cC^E = \cC^{G_k}$ by \Cref{L:eventual t-structure}. By construction, $A\in \cA$ and $H^i(A) = 0$ for all $i\ne 0$.  Write $C = \Cone(A''\to A)$. We consider the long exact sequence $\cdots\to H^0(A'')\to H^0(A) \to H^0(C)\to \cdots$ associated to $A''\to A\to C$. 

    $C\in \cC_{\prec E}^E = \cC_{\prec G_k}^{G_k}$ and thus $H^i(C) \in \cC^{G_k}_{\prec G_k} \cap \cA$ for all $i$. As $\cC_{\prec G_k}^{G_k} \cap \cA$ is closed under quotients, being a Serre subcategory of $\cA$ by \Cref{L:eventual_heart_serre_subcategory}, it follows that for all $i\ne 0$, $H^i(A'') \in \cC_{\prec G_k}^{G_k}$. So, if $G_\ell$ is a limit HN factor of $A''$ with $G_\ell \sim E \sim G_k$, then $G_\ell$ must lie entirely in $\cA$. In particular, $G_k\in \cA$ and it follows that $a=0$. Therefore, $0 = \phi^{G_k}(G_k) \ge \phi^{G_k}(A)$. It also follows that $\phi^E(G_k) \ge \phi^E(A)$. Thus, $\phi^E(G_1)>\cdots > \phi^E(G_k) \ge \phi^E(A) > \phi^E(B)$. In particular, by the properties of the slicing $\cP^E$, $\Hom(A'',B) = 0$. 
\end{proof}

\begin{proof}[Proof of \Cref{T:prestabilityonquotient}]
We define pre-stability conditions on $\cC_{\preceq E}^E/\cC_{\prec E}^E$ and then transport them to $\cC_{\preceq E}/\cC_{\prec E}$ using \Cref{L:Jequiv}.  

$\forall A\in \lss$ such that $A\sim E$, define $Z_E(A)=\exp(\lim_{t\to\infty} \ell_t(A/E))$ and  $\phi_E(A) = \lim_{t\to\infty}\phi_t(A)-\phi_t(E)$, both of which exist by \Cref{D:exit_sequence}\ref{I:semistable_difference}. By \Cref{L:logZcomp}, $Z_E(A) = \lim_{t\to\infty} Z_t(A)/Z_t(E)$, and thus extends by additivity to an element of $\Hom(K_0(\cC_{\preceq E}^E),\bC)$. For all limit semistable $A\in \cC_{\prec E}^E$, $Z_E(A) = 0$ and so for all $X\in \cC_{\prec E}^E$, $Z_E(X) = 0$. There is an exact sequence $K_0(\cC_{\prec E}^E)\to K_0(\cC_{\preceq E}^E)\to K_0(\cC_{\preceq E}^E/\cC_{\prec E}^E)\to 0$ (see \cite[Thm. 5.1]{Kellerdgcat}) and therefore $Z_E$ descends to $\Hom(K_0(\cC_{\preceq E}^E/\cC_{\prec E}^E),\bC)$.

Note that $\lss(\phi_t(E) + \phi) \subset \cC_{\preceq E}^E$; define $\cP_E(\phi)$ to be its essential image in $\cC_{\preceq E}^E/\cC_{\prec E}^E$ for each $\phi$. For $A\sim E$, $\lim_{t\to\infty} \phi_t(A)-\phi_t(E)=\phi=\phi_E(A)$. One has
$Z_E(A) = \lvert Z_E(A)\rvert \exp(i\pi \phi_E(A))$, as needed. $\cC_{\preceq E}^E\to \cC_{\preceq E}^E/\cC_{\prec E}^E$ is exact, so $\phi^E(A[1]) = \phi^E(A) + 1$.

Any object in $\cC_{\preceq E}^E/\cC_{\prec E}^E$ admits a lift to some object $F \in \cC_{\preceq E}^E$. If one starts with a limit HN filtration of $F$ and deletes every step in the filtration $F_i$ such that $\Cone(F_{i-1} \to F_i) \in \cC_{\prec E}^E$, the resulting filtration projects to an HN filtration in $\cC_{\preceq E}^E/\cC_{\prec E}^E$ for the original object.\endnote{Let $F\in \cC_{\preceq E}^E$ be given with limit HN filtration $0 = F_0 \to F_1\to\cdots \to F_m = F$ and factors $\{G_i = \Cone(F_{i-1}\to F_i)\}$. Beginning with $i=1$, if $G_i\in \cC_{\prec E}^E$, then $F_{i-1}\to F_i$ is an isomorphism in $\cC_{\preceq E}^E/\cC_{\prec E}^E$. So, remove $E_i$ and use the triangle constructed using the composite morphism: $F_{i-1}\to F_{i+1}\to G_{i+1}$. Proceed until all such $F_i$ and $G_i$ are removed. Thus, up to reindexing we may assume $F_i \sim E$ for all $1\le i \le m$. Then, $\phi_t(G_1)>\cdots>\phi_t(G_m)$ $\forall t\gg0$ and $\phi_E(G_1)>\cdots>\phi_E(G_m)$, giving HN filtrations for $\lss_E$.}

Suppose given $A,B\in \lss$ with $A\sim B \sim E$ such that $\phi_E(A) > \phi_E(B)$. An element of $\Hom_{\cC_{\preceq E}^E/\cC_{\prec E}^E}(A,B)$ is represented by diagram $A\leftarrow A'\to B$
in $\cC_{\preceq E}^E$ with $\Cone(A'\to A)\in \cC_{\prec E}^E$, up to a natural equivalence relation (see \cite{NeemanTriangulated}*{Defn. 2.1.11}).\endnote{Let $\cC$ denote a triangulated category and $\cD$ a thick subcategory. Denote by $\cQ = \cC/\cD$ the Verdier quotient category. In loc. cit., morphisms in the Verdier quotient category are defined by ``roof'' diagrams as follows. $\Hom_{\cQ}(A,B)$ consists of equivalence classes of diagrams $A\xleftarrow{f} A' \to B$ where the arrows are morphisms in $\cC$ and $\Cone(f)\in \cD$. Two such diagrams $A\xleftarrow{f} A' \to B$ and $A\xleftarrow{g} A'' \to B$ are declared equivalent if there is a commutative diagram:
\[
\begin{tikzcd}[ampersand replacement=\&]
    \&A'\arrow[dl,"f",swap]\arrow[dr]\& \\
    A\& A'''\arrow[l,"h",swap]\arrow[r]\arrow[u]\arrow[d] \& B\\
    \& A''\arrow[ul,"g"]\arrow[ur] \&
\end{tikzcd}
\]
in $\cC$ with $\Cone(h) \in \cD$.} 
By \Cref{L:PhaseConvFactors}, there exists $A''\in \cC_{\preceq E}^E$ with a morphism $A''\to A$ such that $\Cone(A''\to A)\in \cC_{\prec E}^E$ and $\Hom_{\cC}(A'',B) = 0$. Hence $A \leftarrow A' \to B$ and $A\leftarrow A'' \to B$ are equivalent as morphisms in $\cC_{\preceq E}^E/\cC_{\prec E}^E$,\endnote{The diagram witnessing the equivalence is 
\[
\begin{tikzcd}[ampersand replacement=\&]
    \&A'\arrow[dl]\arrow[dr]\& \\
    A\& A''\arrow[l]\arrow[r]\arrow[u]\arrow[d,equal] \& B\\
    \& A''\arrow[ul]\arrow[ur] \&
\end{tikzcd}
\]
where the arrow $A''\to B$ is the composite $A''\to A'\to B$.} and the latter is equivalent to $0$. This implies that $\Hom_{\cC_{\preceq E}/\cC_{\prec E}}(A,B) = 0$. 

By \Cref{L:Jequiv}, the pre-stability condition $(Z_E,\cP_E)$ on $\cC_{\preceq E}^E/\cC_{\prec E}^E$ induces one on $\cC_{\preceq E}/\cC_{\prec E}$ also denoted $(Z_E,\cP_E)$ by abuse of notation. Any limit semistable $F\sim E$ is in the image of the inclusion $\cC_{\preceq E}^E \hookrightarrow \cC_{\preceq E}$ and so $Z_E(F) = \exp(\lim_{t\to\infty} \ell_t(F/E))$, whence (2) follows. By the diagram defining $J$, (1) follows also. 

Finally, if $E\sim E'$ then $\cC_{\preceq E} = \cC_{\preceq E'}$ and $\cC_{\prec E} = \cC_{\prec E'}$ so $\cC_{\preceq E}/\cC_{\prec E} = \cC_{\preceq E'}/\cC_{\prec E'}$. As $E\sim E'$, $\lim_{t\to\infty}\ell_t(E'/E) =: z_{E'/E}$ exists in $\bC$ and equals $\lim_{t\to\infty} \logZ_t(E') -\logZ_t(E)$ by \Cref{L:logZcomp}. Consider $z_{E'/E}\cdot \sigma_{E'} = (W,\cQ)$; one can check that $W = Z_E$. By definition, $\cP_E(\phi)$ consists of those $F\in \lss$ such that $\lim_t \phi_t(F) - \phi_t(E) = \phi$ and similarly for $\cP_{E'}(\phi)$. $\cQ(\phi) = \cP_{E'}(\phi - (\lim_{t\to\infty} \phi_t(E') - \phi_t(E)))$ and in particular consists of all $F \in \lss$ such that $\lim_{t}\phi_t(F) - \phi_t(E) = \phi$. I.e., $\cQ = \cP_E$ as claimed.
\end{proof}

\begin{rem}
\label{R:phantoms}
We conclude with a pair of remarks:
\begin{enumerate}
    \item We have introduced a pair of similar looking slicings, $\cP^E$ in \Cref{L:eventual t-structure}, and $\cP_E$ in \Cref{T:prestabilityonquotient}. Note that $\cP^E$ is defined on $\cC^E$, while $\cP_E$ is defined on $\cC_{\preceq E}^E/\cC_{\prec E}^E$ and thus on $\cC_{\preceq E}/\cC_{\prec E}$. Also, in the definition of $\cP^E$ we use $\phi_E(A) = \limsup_{t\to\infty} \phi_t(A) - \phi_t(E)$ since the limit of $\phi_t(A) - \phi_t(E)$ is not defined unless $A\sim E$. When $\sim$ and $\isim$ are equivalent relations, the natural functor $\cC^E\to \cC_{\preceq E}^E/\cC_{\prec E}^E$ is an equivalence identifying $\cP^E$ and $\cP_E$ so that $(Z_E,\cP_E)$ defines a pre-stability condition on $\cC^E$. \vspace{2mm}
    \item Since $Z_E(E) = 1$, $K_0(\cC_{\preceq E}/\cC_{\prec E})\otimes \bQ \ne 0$. In particular, the subquotient categories obtained from quasi-convergent paths are never phantom categories --- cf. \cite{GorchOrlov}.
\end{enumerate}
\end{rem}

\subsection{Numerical quasi-convergent paths}
\label{S:numerical}

Fix a quasi-convergent path $\sigma_\bullet$ in $\Stab_\Lambda(\cC)$. In this subsection we investigate conditions under which the central charge $Z_E : K_0(\cC_{\preceq E} / \cC_{\prec E}) \to \bC$ of the pre-stability condition from \Cref{T:prestabilityonquotient} factors through the subquotient group,
\[
\Lambda_E := v(\cC_{\preceq E})/ \{\alpha \in v(\cC_{\preceq E}) | m\alpha \in v(\cC_{\prec E}) \text{ for some } m \in \bZ\} 
\]
and when the resulting pre-stability condition has the support property.

\begin{defn}
\label{D:numerical} The quasi-convergent path $\sigma_\bullet$ is called \emph{numerical} if for any $E_1,\ldots,E_n \in \lss$ that are pairwise non-equivalent with respect to $\isim$, the subgroups $v(\cC^{E_i}) \subset \Lambda$ are linearly independent over $\bQ$.
\end{defn}

\begin{lem}
\label{L:linearalgebra}
    If $\sigma_\bullet$ is numerical, then
    \begin{enumerate}
        \item $\Lambda = \bigoplus_{F \in \lss/{\isim}} v(\cC^F)$ and $\lss/{\isim}$ is finite.
    \end{enumerate}
    For all $E,F \in \lss$
    \begin{enumerate}
        \item[(2)]  $v(\cC_{\preceq E}^F) = v(\cC_{\preceq E}) \cap v(\cC^F)$ and $v(\cC_{\prec E}^F) = v(\cC_{\prec E}) \cap v(\cC^F)$; and \vspace{2mm}
        \item[(3)] if $E \isim F$, the natural map induces an isomorphism 
        \[
        v(\cC_{\preceq E}) \cap v(\cC^F) / v(\cC_{\prec E}) \cap v(\cC^F) \cong v(\cC_{\preceq E})/v(\cC_{\prec E}).
        \]
    \end{enumerate}
\end{lem}

\begin{proof}
By \Cref{D:numerical}, $\# (\lss/{\isim}) \le \dim_{\bQ} \Lambda_{\bQ} <\infty$. Therefore, the semiorthogonal decomposition of \Cref{P:ImSOD} is finite. The decomposition $K_0(\cC) = \bigoplus_{F\in \lss/{\isim}}K_0(\cC^F)$ combined with surj\-ectivity of $v: K_0(\cC)\twoheadrightarrow \Lambda$ implies that the subgroups $\{v(\cC^F)\}_{F \in \lss/{\isim}}$ generate $\Lambda$. Linear indep\-endence is by \Cref{D:numerical} and (1) follows.

$v(\cC_{\preceq E}^F) \subseteq v(\cC_{\preceq E})\cap v(\cC^F)$ is automatic. Given $x\in v(\cC_{\preceq E}^F)$, write $x = v(G)$ for $G\in \cC_{\preceq E}$. By (1), $v(G) \in v(\cC^F)$ implies $G\in \cC^F$. So, $v(\cC_{\preceq E}^F) = v(\cC_{\preceq E})\cap v(\cC^F)$. $v(\cC_{\prec E}^F) = v(\cC_{\prec E}) \cap v(\cC^F)$ is analogous. 

For (3), $v(\cC_{\preceq E})$ and $v(\cC_{\prec E})$ contain $v(\lss(\iprec F))$, so one has $v(\cC_{\preceq E}) = v(\lss(\iprec F))\oplus (v(\cC_{\preceq E})\cap v(\cC^F))$ and $v(\cC_{\prec E}) = v(\lss(\iprec F)) \oplus (v(\cC_{\prec E})\cap v(\cC^F))$. The claim follows.
\end{proof}

\begin{defn}
\label{D:supportpropertypath}
    A numerical quasi-convergent path $\sigma_\bullet$ in $\Stab_\Lambda(\cC)$ satisfies the \emph{support property} if 
    $\forall E\in \lss$ and for some (equivalently any) norm $\lVert \:\cdot\:\rVert_E$ on $\Lambda_E \otimes \bR$, $\exists \epsilon_E > 0$ such that $\forall F\in \lss$ with $F\sim E$
\begin{equation*}
    \lim_{t \to \infty} \frac{\lvert Z_t(F)\rvert}{\lvert Z_t(E)\rvert} \ge \epsilon_E \lVert v(F)\rVert_E.
\end{equation*}
 
\end{defn}

\begin{thm}\label{T:stabilityonquotient}
Suppose $\sigma_\bullet$ is a numerical quasi-convergent path in $\Stab_\Lambda(\cC)$.

\begin{enumerate}
    \item The central charge $Z_E$ from \Cref{T:prestabilityonquotient} factors through $\Lambda_E$ for all $E\in \lss$. In particular, $\Lambda_E \ne 0$ for all $E\in \lss$. \vspace{2mm}

    \item $\sigma_\bullet$ satisfies the support property if and only if $\forall E\in \lss,$ the pre-stability condition $\sigma_E$ from \Cref{T:prestabilityonquotient} satisfies the support property with respect to $\Lambda_E$.
\end{enumerate}
\end{thm}

\begin{proof}
    For (1), $\sigma_E$ on $\cC_{\preceq E}/\cC_{\prec E}$ is constructed by defining a pre-stability condition on $\cC_{\preceq E}^E/\cC_{\prec E}^E$ and then using $J: \cC_{\preceq E}/\cC_{\prec E}\xrightarrow{\sim} \cC_{\preceq E}^E/\cC_{\prec E}^E$ of \Cref{L:Jequiv} to transport it. $Z_E$ factors through $v(\cC_{\preceq E})/v(\cC_{\prec E})$ by commutativity of
    \[
    \begin{tikzcd}
        K_0(\cC_{\preceq E}/\cC_{\prec E})\arrow[r,"\overline{v}"] & v(\cC_{\preceq E})/v(\cC_{\prec E})\arrow[dr,dashed,"Z_E"]&\\
        K_0(\cC_{\preceq E}^E/\cC_{\prec E}^E)\arrow[r,"\overline{v}"] \arrow[u, "\sim" {anchor=south, rotate=90}] & v(\cC_{\preceq E}^E)/v(\cC_{\prec E}^E)\arrow[r,"Z_E\:\:\:\:"] \arrow[u, "\sim" {anchor=north, rotate=90}] &\bC.
    \end{tikzcd}
    \]
    Note the double usage of $Z_E$. The second vertical isomorphism is by parts (2) and (3) of \Cref{L:linearalgebra}. In both cases, $\overline{v}$ denotes $v$ composed with the quotient map. Finally, factorization of $Z_E$ through the torsion free part of $v(\cC_{\preceq E})/v(\cC_{\prec E})$, $\Lambda_E$, is immediate from the fact that $Z_E$ is valued in $\bC$.

    For (2), fix a norm $\lVert \:\cdot\:\rVert$ on $\Lambda_E$ for each $E$. $\sigma_E$ has the support property for $\Lambda_E$ if and only if there exists $\epsilon_E>0$ such that for all $F\in\lss$ with $E\sim F$ one has $\lvert Z_E(F)\rvert/\lVert v(F)\rVert \ge \epsilon_E$ which is equivalent to \Cref{D:supportpropertypath}. 
\end{proof}

\begin{cor}
    If $\sigma_\bullet$ is numerical then $\#(\lss/{\sim})\le \dim \Lambda_{\bQ} <\infty$.
\end{cor}

\begin{proof}
    By \Cref{L:linearalgebra}, $\lss/{\isim}$ is finite. Given $[F]\in \lss/{\isim}$, $\sim$ induces an equivalence relation on $[F]$, regarded as a subset of $\lss$ and $\lss/{\sim} = \bigsqcup_{F\in \lss/{\isim}} [F]/{\sim}$. On the other hand, because $\dim(\Lambda_E)_\bQ \geq 1$ for all $E$, $\#([F]/{\sim}) \le \sum_{E\in [F]/{\sim}} \dim (\Lambda_E)_{\bQ}  = \dim v(\cC^F)_{\bQ}$ by \Cref{T:stabilityonquotient}. So, $\#(\lss/{\sim}) \le \sum_{E\in \lss/{\sim}} \dim (\Lambda_E)_{\bQ} = \dim \Lambda_{\bQ}$. 
\end{proof}

The remainder of the section is devoted to showing that many paths considered in practice are numerical.

\begin{prop}
\label{P:numericalforK}
    Let $\cC$ denote a triangulated category with $0 <\rank K_0(\cC)<\infty$. Let $\Stab(\cC)$ denote the space of stability conditions satisfying the support property with respect to $K_0(\cC)\to K_0(\cC)_{\rm{tf}}$. Every quasi-convergent path in $\Stab(\cC)$ is numerical.
\end{prop}

\begin{proof}
    Numericity follows from additivity: given a semiorthogonal decomposition $\cC = \langle \cC_1,\ldots, \cC_n\rangle$, one has $K_0(\cC) = \bigoplus_{i=1}^n K_0(\cC_i)$.
\end{proof}

\begin{ex}
    Let $A$ be a finite dimensional algebra over a field $k$ and $\cD = \mathrm{D}^{\mathrm{b}}(\rm{mod}\:A)$ its bounded derived category of finite dimensional modules. $K_0(\cD)$ is free of finite rank on the classes of the simple finite dimensional $A$-modules. Bridgeland observed that $\Stab(\cD)$ is always nonempty \cite[Ex 5.5]{Br07}. \Cref{P:numericalforK} implies that every quasi-convergent path in $\Stab(\cD)$ is numerical.
\end{ex}

For a dg-category $\cD$ over $\bC$, Blanc constructs its topological K-theory spectrum $\bf{K}^{\rm{top}}(\cD)$ \cite{BlancTopK} along with a canonical morphism of spectra $\bf{K}(\cD)\to \bf{K}^{\rm{top}}(\cD)$, where $\bf{K}(\cD)$ denotes the algebraic K-theory spectrum of \cite{Schlichting06}.
\begin{prop} \label{P:top_numerical}
If $\Lambda := \im(K_0(\cD) \to K_0^{\rm{top}}(\cD))_{\rm{tf}}$ has finite rank, then any quasi-convergent path in $\Stab_\Lambda(\cD)$ is numerical.
\end{prop}

\begin{proof} 
     Using \cite{BlancTopK}*{Thm 1.1}, it follows that $K_0^{\rm{top}}$ is an additive invariant of semiorthogonal decomp\-ositions.\endnote{We give a proof of this in the finite semiorthogonal decomposition case, as this is our intended application. It suffices to consider $\cC = \langle \cA,\cB\rangle$ of a pre-triangulated dg-category $\cC$ over $\bC$. There is an associated exact sequence of dg-categories $0\to \cA \to \cC\to \cC/\cA \to 0$. $\pi:\cC\to \cC/\cA$ admits a section $s:\cC/\cA\to \cC$ given by $\cC/\cA\xrightarrow{\sim} \cB\hookrightarrow \cC$. \cite{BlancTopK}*{Thm. 1.1(c)} gives a distinguished triangle in the homotopy category of spectra
    \[
    \bf{K}^{\rm{top}}(\cA)\to \bf{K}^{\rm{top}}(\cC)\to \bf{K}^{\rm{top}}(\cC/\cA)
    \]
    and passing to the long exact sequence of homotopy groups one has 
    \[
    \cdots \to K_{1}^{\rm{top}}(\cC/\cA) \to K_0^{\rm{top}}(\cA)\to  K_0^{\rm{top}}(\cC)  \to K_0^{\rm{top}}(\cC/\cA) \to \cdots
    \]
    however the splitting $\pi \circ s = \id_{\cC/\cA}$ induces a splitting $K_i^{\rm{top}}(\cC) = K_i^{\rm{top}}(\cA) \oplus K_i^{\rm{top}}(\cB) \cong K_i^{\rm{top}}(\cA) \oplus K_i^{\rm{top}}(\cC/\cA)$ for each $i$ whence the result follows. So, given $\cC = \langle \cC_1,\ldots, \cC_n\rangle$ one has a direct sum decomposition $K_\ell(\cC) = \bigoplus_{i=1}^n K_\ell(\cC_i)$ $\forall\: \ell$. 
    } So, for $\cD = \langle \cC_1,\ldots, \cC_n\rangle$, one has $\bigoplus_{i=1}^n K_0(\cC_i) = K_0(\cD)$ and $\bigoplus_{i=1}^n K_0^{\rm{top}}(\cC_i) = K_0^{\rm{top}}(\cD)$. Furthermore, $K_0(\cD)\to K_0^{\rm{top}}(\cD)$ maps $K_0(\cC_i) \to K_0^{\rm{top}}(\cC_i)$,\endnote{This follows from the functoriality statement in \cite{BlancTopK}*{Thm. 1.1(d)}. Indeed, $\cC_i\hookrightarrow \cC$ is a functor of $\bC$-linear dg-categories and consequently there is a commutative square 
    \[
    \begin{tikzcd}[ampersand replacement=\&]
    K_0(\cC_i)\arrow[d]\arrow[r]\&K_0(\cC)\arrow[d]\\
    K_0^{\rm{top}}(\cC_i)\arrow[r]\&K_0^{\rm{top}}(\cC).
    \end{tikzcd}\]
    } 
    so there is an induced decomposition $\Lambda = \bigoplus_{i=1}^n \im(K_0(\cC_i)\to K_0^{\rm{top}}(\cC_i))_{\rm{tf}}$, as needed. The claim follows.\endnote{Suppose given a quasi-convergent path $\sigma_\bullet$ in $\Stab(\cD)$ with associated semiorthogonal decomposition $\langle \cC^F: F \in \lss/{\isim}\rangle$. Given limit semistable objects $F_1,\ldots F_n$ representing distinct classes in $\lss/{\isim}$, one has $K_0(\cD) = \bigoplus_{i=1}^nK_0(\cC^{F_i}) \oplus K_0(\cD)$ where $\cD = \langle \cC^G: G \not\isim F_1,\ldots, F_n\rangle$; a similar decomposition holds for $K^{\rm{top}}_0(\cD)$. Tensoring with $\bQ$ yields the result.} 
\end{proof}

\begin{ex}\label{E:numerical}
When $\cD = \DCoh(X)$ for a smooth complex projective variety $X$, one often considers stability conditions whose central charge factors through $\ch:K_0(X) \to H^*_{\rm{alg}}(X)_{\rm{tf}}$, where $H^*_{\rm{alg}}(X):= \im(\ch:K_0(X) \to H^*(X;\bZ))$ \cite{BMT}. $\Stab(X)$ denotes the corresponding space of stability conditions. $\bf{K}^{\rm{top}}(\cD)$ is equivalent to the usual topological K-theory spectrum of $X$, and taking $\pi_0$ of Blanc's map $\bf{K}(\cD) \to \bf{K}^{\rm{top}}(\cD)$ recovers the canonical map $K_0(X)\to K_0^{\rm{top}}(X)$. We also have a commutative diagram
\begin{equation*}
    \begin{tikzcd}
        K_0(X)\arrow[r,"\ch"]\arrow[d] & H^*_{\rm{alg}}(X)\arrow[d]\\
        K_0^{\rm{top}}(X)_{\bQ}\arrow[r,"\ch"]\arrow[r,"\cong",swap]& H^{\rm{even}}(X;\bQ),
    \end{tikzcd}
\end{equation*}
which combined with \Cref{P:top_numerical} shows that any quasi-convergent path in $\Stab(X)$ is numerical.
\end{ex}

\section{Gluing stability conditions}

\subsection{Preliminaries on homological algebra}

We establish a pair of homological algebra results for use in \S \S 3.2-3.3. If $\cB$ is a pre-triangulated dg-category with a bounded t-structure, we denote by $\cB^\heart$ its heart.

\begin{prop}
\label{P:Boundedtamp}
Let $\cB$ and $\cC$ be idempotent complete pre-triangulated dg-categories with bounded $t$-structures. If $\cB$ is smooth and proper, then any exact functor $\psi:\cB\to \cC$ has bounded t-amplitude.
\end{prop}

\begin{proof}
By \cite[Cor.~2.13]{ToenVaquie}, $\cB$ admits a classical generator $G \in \cB$ in the sense of \cite{BvdB02}, and by replacing $G$ with its homology, we may assume that $G \in \cB^\heart$. Because $\cB$ is smooth, the identity functor $\id : \cB \to \cB$ lies in the idempotent complete pre-triangulated closure of the functor $G \otimes_k \RHom_\cB(G,-)$ in the $\infty$-category $\Fun_k^{\rm{ex}}(\cB,\cB)$ of exact $k$-linear functors.\endnote{Derived Morita theory \cite{DerivedMorita} induces an equivalence between exact functors $\cB \to \cB$ and $\cB \otimes_k \cB^{\rm op}$-modules, under which the identity functor corresponds to the diagonal $\cB \otimes_k \cB^{\rm op}$-module. By definition, if $\cB$ is smooth, this bimodule lies in the idempotent complete pre-triangulated closure of $\cB \otimes_k \cB^{\rm op}$. Hence the identity functor $\id : \cB \to \cB$ lies in the idempotent complete pre-triangulated closure of the functor $G \otimes_k \RHom_\cB(G,-)$.} It follows that $\psi \cong \psi \circ \id_{\cB}$ lies in the idempotent complete pre-triangulated closure of the functor
\[
M \mapsto \psi(G) \otimes_k \RHom_\cB(G,M)
\]
in $\Fun_k^{\rm {ex}}(\cB,\cC)$.

This reduces the claim to showing that the functor $\cB \to k\Mod$ taking $M \mapsto \RHom_\cB(G,M)$ has uniformly bounded $t$-amplitude. Since $G \in \cB^\heart$, for any $M \in \cB^\heart$ one has $H^i(\RHom_\cB(G,M)) = 0$ for $i<0$. On the other hand, because $\cB$ is smooth and proper it admits a Serre functor $S : \cB \to \cB$, and we have
\[
H^i(\RHom_\cB(G,M)) = H^{-i}(\RHom_\cB(M,S(G))^*).
\]
If $k$ is the degree of the lowest non-vanishing cohomology object of $S(G)$ in the $t$-structure on $\cB$, then the right hand side vanishes for $i>-k$ whenever $M \in \cB^\heart$.
\end{proof}

\begin{prop}
\label{P:hearthombound}
Suppose $\cC$ is a smooth, proper, and idempotent complete pre-triangulated dg-category with $\cC=\langle \cC_1,\cC_2\rangle$. Suppose $\cC_1$ and $\cC_2$ are equipped with bounded t-structures. There exists an $m\in \bb{Z}$ such that $\Hom^{\le m}_{\cC}(\cC_1^\heart,\cC_2^\heart)=0.$
\end{prop}

\begin{proof}
$\cC$ is saturated and hence so is $\cC_2$. In particular, $i_2:\cC_2\to \cC$ admits a left adjoint $L_2$. Take $X\in \cC_1^\heart$ and $Y\in \cC_2^\heart$. Adjunction gives a natural isomorphism $\Hom_{\cC}(X,Y) = \Hom_{\cC_2}(L_2i_1(X),Y).$

$\cC_1$ is smooth and proper, so by \Cref{P:Boundedtamp} the functor $L_2\circ i_1:\cC_1\to \cC_2$ has bounded t-amplitude. In particular, $L_2(i_1(\cC_1^\heart))\subset \cC_2^{[-n,n]}$ for some $n\ge 0$. Now, take $m\le -n-1$. One has $\Hom_{\cC}^{m}(X,Y)=\Hom_{\cC_2}(L_2(i_1(X)),Y[m])$ and $\cC_2^\heart[m]\subset \cC_2^{\ge n+1}$. Because $L_2(i_1(X))\in \cC_2^{\le n}$, it follows that $\Hom_{\cC_2}(L_2(i_1(X)),Y[m])=0.$
\end{proof}

\begin{ex}
\label{Ex:examples}
    We give examples where \Cref{P:hearthombound} holds.
    \begin{enumerate} 
        \item For $X$ a smooth projective variety, $\DCoh(X)$ is smooth, proper, and idempotent complete. Smoothness and properness follow from \cite[Prop. 3.31]{Or16}, and idempotent completeness follows for instance from \cite[Prop. 2.1.1]{BvdB02}. \vspace{2mm}
        \item Suppose $A$ is a finite dimensional algebra of finite global dimen\-sion. Let $\rm{D}^{\rm{b}}(\rm{mod}\:A)$ denote the bounded derived category finite dimensional left $A$-modules. $\rm{D}^{\rm{b}}(\rm{mod}\:A)$ is smooth, proper, and idempotent complete. Smoothness and properness are by \cite[\S8]{KS08}, while idempotent completeness follows from \cite[Prop. 3.4]{BNeeman}. As a simple example, $A$ could be the path algebra of an acyclic quiver.
    \end{enumerate}
\end{ex}

\subsection{Gluing constructions revisited}

We will use ideas and results from \cite{CP10}. Stability cond\-itions satisfying the support property are \emph{reasonable} in the sense of \cite{CP10}, so we may apply their results.\endnote{A stability condition is \emph{reasonable} in the sense of \cite{CP10} if
\[
\inf_{E\ne 0\:\sigma\text{-ss}} \lvert Z(E)\rvert > 0.
\]
In our context, stability conditions are assumed to satisfy the support property of \cite{KS08}. Choose a norm $\lVert\:\cdot\:\rVert$ on $\Lambda \otimes \bb{R}$ such that $\lVert v\rVert \ge 1$ for all $v\in \Lambda$. By our choice of norm, the support property implies $0<C\le \lvert Z(v(E))\rvert$. In particular, it implies the reasonable assumption of \cite{CP10}. So, a stronger assumption is implicit and we remove this terminology. For a discussion of this see \cite{BMSAbelian}*{Appendix A}. 
} 

Our stability conditions on $\cC$ have the support property with respect to some fixed homomorphism $v:K_0(\cC)\twoheadrightarrow \Lambda$ to a free Abelian group of finite rank. Throughout this section, we consider semiorthogonal decompositions $\cC = \langle \cC_1,\ldots, \cC_n\rangle$ together with a splitting $\Lambda = \bigoplus_{i=1}^n \Lambda_i$ such that $v(K_0(\cC_i)) = \Lambda_i$ for each $i$. For each $j$, $\iota_j:\cC_j\to \cC$ denotes the inclusion functor. To simplify notation, we will denote $\Stab(\cC) := \Stab_\Lambda(\cC)$ and $\Stab(\cC_i) := \Stab_{\Lambda_i}(\cC)$ in this context.

\begin{defn}
\label{D:generalglued}
 Let $\vec{\sigma} = (\sigma_1,\ldots, \sigma_n) \in \prod_{i=1}^n \Stab(\cC_i)$ be given and write $\sigma_i = (Z_i,\cQ_i)$ for each $i$. Write $\cA_i = \cQ_i(0,1]$ and suppose $i<j$ implies $\Hom_{\cC}^{\le 0}(\cA_i,\cA_j) = 0$. We say $\sigma = (Z,\cQ) \in \Stab(\cC)$ is \emph{glued} from $\vec{\sigma}$ if 
\begin{enumerate}
    \item $\cA_\sigma := \cQ(0,1] = [\cA_1,\ldots, \cA_n]$; and \vspace{2mm}
    \item for all $E_j \in \cC_j$, $Z(\iota_j(E_j)) = Z_j(E_j)$.
\end{enumerate}
\end{defn}

When $n=2$, \Cref{D:generalglued} recovers \cite[Defn. p. 571]{CP10}. If $\sigma = (Z,\cQ)$ is glued from $\vec{\sigma}$, we write $\gl(\vec{\sigma}) = \sigma$. $\sigma = \gl(\vec{\sigma})$ can be obtained by composing $n-1$ gluing maps from the $n = 2$ case. As a consequence, \cite{CP10}*{Prop. 2.2} can be applied inductively to show that $\cQ_j(\phi)\subset \cQ(\phi)$ for all $1\le j \le n$ and all $\phi \in \bR$.\endnote{For simplicity, consider the $n=3$ case. By \cite{CP10}*{Prop. 2.2}, since $\Hom^{\le 0}_{\cC}(\cA_1,\cA_2) = 0$, there is a glued stability condition on $\langle \cC_1,\cC_2\rangle$ with underlying heart $\cA_{12} = [\cA_1,\cA_2]$ and central charge given by $Z(\iota_j(E_j)) = Z_j(E_j)$ for all $E_j \in \cC_j$ for $j=1,2$. Any object $X$ of $\cA_{12}$ can be written in the form $A_2\to X \to A_1$ so that $\Hom^{\le 0}_{\cC}(\cA_{12},\cA_3)$ and $0$ and there is a glued stability condition on $\cC = \langle \cC_1,\cC_2,\cC_3\rangle$ with heart $[\cA_1,\cA_2,\cA_3] = \cA$ and central charge as before.}

The pair $(\cA_\sigma,Z_\sigma)$ is determined uniquely by \Cref{D:generalglued}, so we obtain a function $\gl: \cG\to \Stab(\cC)$, where $\cG\subset \prod_{i=1}^n \Stab(\cC_i)$ denotes the locus of \emph{gluable} stability conditions, i.e., tuples for which (1) and (2) of \Cref{D:generalglued} define a stability condition on $\cC$.

\begin{lem}
\label{L:stronglygluedobjects}
    In the above notation, suppose $\sigma \in \Stab(\cC)$ is given such that $\gl(\vec{\sigma}) = \sigma$. If $i<j$ implies $\Hom_{\cC}^{\le 1}(\cA_i,\cA_j) = 0$, then 
    \begin{enumerate} 
        \item $\cQ(\phi) = \bigoplus_{i=1}^n\cQ_i(\phi)$; and \vspace{2mm}
        \item $\forall I\subset \bR$, $\cQ(I) = [\cQ_1(I),\ldots, \cQ_n(I)]$.
    \end{enumerate}
\end{lem}

\begin{proof}
Since $\Hom^{1}_{\cC}(\cA_i,\cA_j) = 0$ for all $i<j$, one has $[\cA_1,\ldots, \cA_n] = \cA_1\oplus \cdots \oplus \cA_n$. For (1), without loss of generality take $\phi \in (0,1]$ so that $X\in \cQ(\phi) \subset [\cA_1,\ldots, \cA_n].$ Write $X = E_1\oplus \cdots \oplus E_n$ with $E_i\in \cA_i$ for each $i$. If some $E_j$ admits a morphism $F_j\to E_j$ from $F_j \in \cQ_j(\phi_j)$ with $\phi_j>\phi$, then $F_j\to X$ destabilizes $X$, as $\cQ_j(\phi_j)\subset \cQ(\phi_j)$. So, the HN factors of each $E_j$ have phase $\le \phi$. If some $E_j$ has a $\sigma_j$-HN factor $G_j$ of phase $\psi <\phi$, then $Z(X) = \sum_i Z(E_i)$ is impossible. Consequently, each $E_i\in \cQ_i(\phi)$. 

For (2), $\cQ(I)$ consists of those objects all of whose HN factors have phase in $I$. So, $\cQ(I) = [\cQ(\phi):\phi \in I].$ $\cQ(\phi) = [\cQ_1(\phi),\ldots, \cQ_n(\phi)]$ by (1) and the result follows.      
\end{proof}

We state a modified version of \cite[Thm. 3.6]{CP10}.

\begin{thm}
\label{T:stronggluing}
Given $\vec{\sigma} = (\sigma_1,\ldots, \sigma_n)\in \prod_{i=1}^n \Stab(\cC_i)$, if for all $i<j$,
\begin{enumerate}
    \item $\Hom_{\cC}^{\leq 1}(\cA_i,\cA_j)  = 0$, and \vspace{2mm}
    \item $\exists a \in (0,1)$ such that $\Hom_{\cC}^{\leq 0}(\cQ_i(a,a+1], \cQ_j(a,a+1]) = 0$,
\end{enumerate}
then $\vec{\sigma} \in \cG$, i.e., it is gluable.
\end{thm}

\begin{proof}
    \cite[Thm. 3.6]{CP10} gives the $n=2$ case, except that the resulting pre-stability condition is \emph{a priori} only reasonable. Write $\Lambda = \Lambda_1\oplus \Lambda_2$ and choose norms $\lVert \:\cdot\:\rVert_i$ on $\Lambda_i\otimes \bb{R}$ for $i=1,2$. Define $\lVert \:\cdot\:\rVert$ on $\Lambda\otimes \bb{R}$ by $\lVert\:\cdot\:\rVert_1\oplus \lVert \:\cdot\:\rVert_2$. By \Cref{L:stronglygluedobjects}, any $E\in \cQ(\phi)$ is of the form $E = E_1\oplus E_2$ where $E_i\in \cQ_i(\phi)$ for $i=1,2$. Then $\lVert v(E)\rVert = \lVert v_1(E_1)\rVert_1 +\lVert v_2(E_2)\rVert_2 \le \min \{C_1,C_2\} \lvert Z(v(E))\rvert,$
    where the $C_i>0$ are the constants given by the support property for $\sigma_i$ for $i=1,2$. The general case follows from induction.\endnote{One first glues $\sigma_1$ and $\sigma_2$ to $\sigma_{12}\in \Stab(\langle \cC_1,\cC_2\rangle)$. The resulting heart is $\cA = [\cA_1,\cA_2]$. One then verifies that conditions (1) and (2) hold for $(\sigma_{12},\sigma_3,\ldots, \sigma_n)$ using \Cref{L:stronglygluedobjects}. Consider $X\in \cA_{12}$, and write it as $X_2\to X \to X_1$, where $X_i \in \cA_i$. Then, for any $Y\in \cA_j$ for $j>2$, one has an exact sequence $\Hom^i(X_1,Y) \to \Hom^{i}(X,Y)\to \Hom^i(X_2,Y)$ for $i\le 1$. Vanishing of the outer terms implies vanishing of the inner term for all $i$. \Cref{L:stronglygluedobjects} implies that $\cQ_{12}(a,a+1] = \left[\cQ_1(a,a+1],\cQ_2(a,a+1] \right]$ and this implies condition (2).}
\end{proof}

Given stability conditions $\sigma = (Z,\cQ)$ and $\tau = (Z',\cQ')$ on any category $\cC$, recall from \cite{Br07} the distance between slicings
\[
d_{\rm{slice}}(\sigma,\tau) := \sup_{0\ne E\in \cC}\big\{\lvert \phi_\cQ^-(E)-\phi_{\cQ'}^-(E)\rvert, \lvert \phi_{\cQ}^+(E)-\phi_{\cQ'}^+(E)\rvert \big\}\in [0,\infty].
\]

\begin{defn}
For $r \geq 1$, we say $\vec{\sigma} = (\sigma_1,\ldots,\sigma_n) \in \prod_{i=1}^n \Stab(\cC_i)$ is \emph{$r$-gluable} if for some $\epsilon>0$, $i<j$ implies
\begin{equation}
\label{E:Homorth}
    \Hom_{\cC}^{\le 0}(\cQ_i(-\epsilon-r,1+r], \cQ_j(-r,1+\epsilon+r))=0.
\end{equation}
We refer to the subset $\cG^s \subset \prod_i \Stab(\cC_i)$ of $1$-gluable points as \emph{strongly gluable}.\endnote{Note that if $\vec{\sigma}$ is $r$-gluable, then it is also $s$-gluable for any $1 \leq s \leq r$, and hence it is strongly gluable.} We define a subset $\cW_{r} \subset \prod_{i=1}^n \Stab(\cC_i)$ for $r>1$ as follows:
\[
\cW_{r} := \left\{ (\tau_i)_{i=1}^n \left| \begin{array}{c} \exists (\sigma_i)_{i=1}^{n } \text{ that is } r\text{-gluable,} \\ \text{and } d_{\rm{slice}}(\sigma_i,\tau_i) < r-1, \forall i  \end{array} \right.\right\} 
\]
Note that $\cW_r \subset \cG^s \subset \cG$ for all $r > 1$ by definition and \Cref{T:stronggluing}.\endnote{The inclusion $\cG^s\subset \cG$ is a special case of \Cref{T:stronggluing}. The inclusion $\cW_r\subset \cG^s$ is proven as follows. Suppose $\vec{\tau}\in \cW_r$ is as above with $\vec{\sigma}$ such that $\max_i d_{\rm{slice}}(\sigma_i,\tau_i)<r-1$. Write $\vec{\tau} = (Z_i,\cQ_i)_{i=1}^n$ and $\vec{\sigma} = (Z_i',\cR_i)_{i=1}^n$. By hypothesis, $\cQ_i(-\epsilon -1,2] \subset \cR_i(-\epsilon - r, 1+r]$ and $\cQ_j(-1,2+\epsilon) \subset \cR_j(-r,1+\epsilon +r)$ and so $\Hom^{\le 1}(\cQ_i(-\epsilon-1,2],\cQ_j(-1,2+\epsilon))$ is immediate from \eqref{E:Homorth} applied to $\vec{\sigma}$.}
\end{defn}

\begin{lem}
\label{L:dslice}
    For any $\vec{\sigma}, \vec{\tau} \in \cG^s$, $d_{\rm{slice}}(\gl(\vec{\sigma}), \gl(\vec{\tau})) = \max_i d_{\rm{slice}}(\sigma_i,\tau_i).$
\end{lem}

\begin{proof}
    For each $i$, denote the slicing of $\sigma_i$ (resp. $\tau_i)$ by $\cQ_i$ (resp. $\cR_i$). Denote the slicing of $\gl(\sigma_i)_{i=1}^n$ (resp. $\gl(\tau_i)_{i=1}^n$) by $\cQ$ (resp. $\cR$). We write $\delta = \max_i d_{\rm{slice}}(\sigma_i,\tau_i)$ and $d = d_{\rm{slice}}(\gl(\sigma_i)_{i=1}^n,\gl(\tau_i)_{i=1}^n).$
    \begin{align*}
        \delta &= \max_{i=1}^n\sup_{0\ne E\in \cC_i}\big\{\lvert \phi^-_{\cQ_i}(E)-\phi^-_{\cR_i}(E)\rvert,\lvert \phi_{\cQ_i}^+(E)-\phi_{\cR_i}^+(E)\rvert\big\}\\
        &\le \sup_{0\ne E\in \cC}\big\{\lvert \phi^-_{\cQ}(E)-\phi^-_{\cR}(E)\rvert,\lvert \phi_{\cQ}^+(E)-\phi_{\cR}^+(E)\rvert\big\} = d
\end{align*}
where the inequality comes from viewing $E\in \cC_i$ as an object of $\cC$.\endnote{The $\cQ$-HN filtration then agrees with the $\cQ_i$-HN filtration and likewise for $\cQ'$} Hence, $\delta \le d$. If $\delta =\infty$, we are done. So, suppose $\delta\in \bb{R}$. Any $E\in \cQ(\phi)$ is of the form $E = \bigoplus_{i=1}^n E_i$ for $E_i \in \cQ_i(\phi)$ by \Cref{L:stronglygluedobjects}. By hypothesis, $E_i\in \cR_i[\phi-\delta,\phi+\delta]$ for all $1\le i \le n$. Because $\cR_i[\phi-\delta,\phi+\delta]\subset \cR[\phi-\delta,\phi+\delta]$ and the latter is extension closed, we have $E\in \cR[\phi-\delta,\phi+\delta]$ also. Thus, by \cite[Lem. 6.1]{Br07} $\delta \ge d$.
\end{proof}

\begin{thm}
\label{P:rgluesurjective}
    For all $r > 1$, both $\cW_r \subset \prod_i \Stab(\cC_i)$ and $\gl(\cW_r) \subset \Stab(\cC)$ are open, and one has
\begin{equation}
\label{E:glueddescription}
    \gl(\cW_r) = \left\{\sigma \in \Stab(\cC) \left| \begin{array}{c} \exists r\text{-gluable } \vec{\tau} \in \prod_i \Stab(\cC_i) \\ \text{s.t. } d_{\rm{slice}}(\sigma, \gl(\vec{\tau})) < r-1 \end{array} \right. \right\}.
\end{equation}
    Furthermore, $\gl : \cW_r \to \gl(\cW_r)$ is a biholomorphism, with inverse given by $(Z,\cQ) \mapsto (Z_i, \cQ_i)_{i=1}^n$, where
    \begin{enumerate}
        \item $Z_i(E) := Z(\iota_i(E))$ for all $E \in \cC_i$; and  \vspace{2mm}
        \item $\cQ_i(\phi) := \cC_i \cap \cQ(\phi)$ for all $\phi \in \bR$.
    \end{enumerate}
    Finally, $\cW_r$ is nonempty if $\cC$ is smooth and proper and $\prod_i \Stab(\cC_i)$ is nonempty.
\end{thm}

\begin{proof} 
For each $r\ge 1$, the condition $\Hom^{\le 0}(\cP_i(-\epsilon -r ,1+r],\cP_j(-r,1+\epsilon +r))$ is an open condition on slicings for each $i<j$. In particular, $\cG^s$ is an open subset of $\prod_{i} \Stab(\cC_i)$. 

Suppose given $\vec{\tau}\in \cW_r$. By definition, there exists an $r$-gluable $\vec{\sigma}$ such that $\max_i d_{\rm{slice}}(\sigma_i,\tau_i) = r-1 - \epsilon$ for some $\epsilon > 0$. Let $U$ be an open neighborhood of $\vec{\tau}$ such that for all $\vec{\eta}\in U$ one has $\max_i d_{\rm{slice}}(\eta_i,\tau_i)<\epsilon$. It follows from the triangle inequality that $U \subset \cW_r$ and hence that $\cW_r$ is open. 

It follows from \cite{Bayer_short}*{Thm. 1.2} that $\Stab(\cC)\to \Hom(\Lambda,\bC)$ given by $(Z,\cQ)\mapsto Z$ is a covering map onto an open subset of $\Hom(\Lambda,\bC)$. This gives local holomorphic coordinates on $\prod_{i=1}^n \Stab(\cC_i)$ and $\Stab(\cC)$ in which the map $\bigoplus_{i=1}^n \Hom(\Lambda_i,\bC)\to \Hom(\Lambda,\bC)$ induced by $\gl$ is $(Z_1,\ldots, Z_n)\mapsto Z_1\oplus \cdots \oplus Z_n$. In particular, $\gl$ is a local biholomorphism and $\gl(\cW_r)$ is open.

Denote the right hand side of \eqref{E:glueddescription} by $\Sigma$. If $\sigma = \gl(\vec{\sigma})$ for $\vec{\sigma}\in \cW_r$, then there exists an $r$-gluable $\vec{\tau}$ such that $\max_i d_{\rm{slice}}(\sigma_i,\tau_i) < r-1$. By \Cref{L:dslice}, $d_{\rm{slice}}(\sigma,\gl(\vec{\tau})) < r-1$ and $\gl(\cW_r)\subseteq \Sigma$ follows. To prove the $\supseteq$ containment, we use the inverse map to $\gl$. 

Suppose given $\sigma = (Z,\cQ)$ such that there exists an $r$-gluable $\vec{\tau}\in \prod_i \Stab(\cC_i)$ and such that $d_{\rm{slice}}(\sigma,\gl(\vec{\tau})) < r-1$. Put $\vec{\tau} = (Z_i',\cQ_i')_{i=1}^n$ and $\gl(\vec{\tau}) = (Z',\cQ')$. For each $i$, let $\cQ_i = \{\cQ_i(\phi)\}_{\phi \in \bR}$ where $\cQ_i(\phi):= \cQ(\phi)\cap \cC_i$. By hypothesis, $\cA := \cQ(0,1] \subset \cQ'(I)$ for $I\subset (-\epsilon -r,r+\epsilon]$. So, $\Hom^{\le 1}(\cQ_i'(I),\cQ_j'(I)) = 0$ for all $i<j$ by the $r$-gluability of $\vec{\tau}$. Therefore, $\cA \subset \cQ'(I) = \bigoplus_{i=1}^n \cQ_i'(I)$.

 We claim $\cQ_i$ defines a slicing on $\cC_i$. The only property not immediate is the existence of HN filtrations for all $X\in \cC_i$.\endnote{Being the intersection of full additive subcategories of $\cC$, $\cQ_i(\phi)$ is also additive. $\cQ_i(\phi)[1] = (\cC_i\cap \cQ(\phi))[1] = \cC_i\cap \cQ(\phi+1)$. Similarly, $\Hom(\cQ(\phi_1) ,\cQ(\phi_2)) = 0$ for $\phi_1>\phi_2$ and this implies the result for $\cQ_i$.} Since $\sigma$-HN filtrations are constructed by concatenating the filtrations in shifts of $\cA$, we may suppose $X\in \cA$. In $\cA$, the $\sigma$-HN filtration is of the form $0 = E_0\subset E_1\subset \cdots \subset E_n = X$ where for each $j$ one has $E_j = E_j^1\oplus \cdots \oplus E_j^n$ for $E_j^i\in \cQ_i'(I)$. However, $\Hom(E_j^k,X) =0$ for all $k\ne i$.\endnote{As $X\in \cC_i$, $\Hom(E_j^k,X)=0$ for $k>i$ by the semiorthogonality condition. On the other hand, for $k<i$, $\Hom(\cQ_k(I),\cQ_i(I)) = 0$.} So, $E_j = E_j^i \in \cC_i$. Thus, all of the subquotients are in $\cC_i$ as well, being cones of morphisms in $\cC_i$, and $\cQ_i$ defines a slicing on $\cC_i$ for each $i$.

$Z_i(v(E)) = Z(v(\iota_i(E)))$, where $v:K_0(\cC)\twoheadrightarrow \Lambda$ is the fixed surjection, so $Z_i$ factors through $K_0(\cC_i) \twoheadrightarrow \Lambda_i$. $(Z_i,\cQ_i)$ satisfies the support property for $\Lambda_i$ because $\cQ_i(\phi)\subset \cQ(\phi)$ for each $\phi$.\endnote{Let $\lVert \:\cdot\:\rVert$ denote a fixed norm on $\Lambda_{\bR}$. The support property is equivalent to the condition that 
\[
    \inf_{E\in \cQ} \frac{\lvert Z_{\sigma'}(E)\rvert}{\lVert v(E)\rVert} \ge C 
\]
for some $C > 0$. Restricting $\lVert \:\cdot\:\rVert$ gives a norm on $\Lambda_{i,\bR}$ for each $i$ and the inclusion $\cQ_i(\phi)\subset \cQ(\phi)$ for each $i$ implies 
\[
    \inf_{E\in \cQ_i} \frac{\lvert Z_{\sigma'}(E)\rvert}{\lVert v(E)\rVert} \ge C
\]
as desired.} 
In particular, $(Z,\cQ)\mapsto (Z_i,\cQ_i)_{i=1}^n$ defines a map $u:\Sigma \to \prod_{i}\Stab(\cC_i)$. One can verify that given $\sigma \in \Sigma$, one has $u(\sigma) \in \cG^s$ and also that $u$ and $\gl$ are mutually inverse. Furthermore, $u(\Sigma) \subset \cW_r$ and it follows that $\Sigma \subset \gl(\cW_r)$, and thus we have $\Sigma = \gl(\cW_r)$, as claimed.

Finally, if $\cC$ is smooth and proper and $\prod_i \Stab(\cC_i)$ is nonempty, consider $\vec{\sigma} = (\sigma_1,\ldots, \sigma_n)$ and define $\vec{\sigma}_t = (\sigma_{1,t},\ldots, \sigma_{n,t})$ for all $t\ge 0$, where $\sigma_{k,t} = e^{ik\pi t}\cdot \sigma_k$ for each $1\le k \le n$. The argument of the proof of \Cref{L:Limitgluing} shows that for any fixed $r\ge 1$, there exists $t(r)$ such that $\vec{\sigma}_t$ is $r$-gluable for all $t\ge t(r)$. In particular, $\cW_r$ is nonempty. 
\end{proof}

\subsection{Gluing and quasi-convergence} 

We show that certain quasi-convergent paths arise from gluing for all $t$ sufficiently large. We also prove a sort of converse to \Cref{P:ImSOD}, showing that under conditions satisfied in practice polarizable semiorthogonal decompositions always arise from a quasi-convergent path in $\Stab(\cC)$.

\subsubsection{Many quasi-convergent paths are eventually glued}

\begin{setup}
\label{S:eventuallyglued}
    Let $\cC$ be a smooth and proper idempotent complete pre-triangulated dg-category and fix a quasi-convergent path $\sigma_\bullet:[0,\infty)\to \Stab(\cC)$ such that $\sim$ and $\isim$ are equivalent relations on $\cP := \cP_{\sigma_\bullet}$. We fix the following notation:
    \begin{itemize}
        \item $S = \{E_1,\ldots, E_n\}\subset \lss$ is a subset such that $S\to \lss/{\isim}$ is a bijection and $i<j \Rightarrow E_i\iprec E_j$. \vspace{2mm}
        \item $\cC_{i} := \cC^{E_i}$ for each $i$. \vspace{2mm}
        \item $\tau_{i,t} := \logZ_t(E_i) \cdot \sigma_{E_i}$, where $\sigma_{E_i}$ is the pre-stability condition on $\cC_i$ from \Cref{T:prestabilityonquotient} and \Cref{R:phantoms}(1). \vspace{2mm}
        \item $\vec{\tau}_t = (\tau_{1,t},\ldots,\tau_{n,t}) : [0,\infty) \to \prod_i \Stab(\cC_i)$.
    \end{itemize}
\end{setup}

We show that subject to mild hypotheses, $\sigma_t$ is in the image of the gluing map for sufficiently large $t$.

\begin{lem}
\label{L:Limitgluing}
In \Cref{S:eventuallyglued}, $\forall r \ge 1$ there exists $t(r)$ such that $t\ge t(r)$ implies that $\vec{\tau}_t$ is $r$-gluable.
\end{lem}

\begin{proof}
Put $\cQ_{E_i}(0,1] := \cA_i$ for each $i$. By \Cref{P:hearthombound}, for all pairs $i<j$ there exists an $n_{ij} \in \bZ$ such that $\Hom^{\le n_{ij}}(\cA_i,\cA_j) = 0$. Set $n = \max_{i<j} n_{ij}$. Let $\phi \in (-\epsilon-r,1+r +\epsilon) = I$ and $k\in \bZ$ with $k\le 1$ be given. For each $i$, $\cQ_{i,t}(\phi) = \cQ_{E_i}(\phi - \phi_t(E_i))$ and so there is a unique $n_i(t,\phi)\in \bZ$ such that $\cQ_{i,t}(\phi)[n_i(t,\phi)]\subset \cQ_{E_i}(0,1]$. 

For $i<j$, $\lim_t \phi_t(E_j) - \phi_t(E_i) = \infty$ and so $\lim_{t\to\infty} n_j(t,\phi) - n_i(t,\phi) = \infty$. Choose $t(r)$ sufficiently large that $t\ge t(r)$ implies $\max_{\phi \in I, i<j} \{k + n_i(t,\phi) - n_j(t,\phi) \} \le n$. As a consequence, for all $i<j$ and $\phi_i,\phi_j \in I$, $\Hom^{\le n}(\cA_{i},\cA_j) = 0$ implies that $\Hom^k(\cQ_{i,t}(\phi_i),\cQ_{j,t}(\phi_j)) = 0$ and so \eqref{E:Homorth} holds. 
\end{proof}

\begin{defn} \label{D:rho}
Let  $\rho(t) = \sup\{r-1 | \vec{\tau}_t \text{ is }r\text{-gluable}\}.$
\end{defn}

If $\rho(t)>0$, so that $\vec{\tau}_t$ is strongly gluable, and $d_{\rm{slice}}(\sigma_t, \gl(\vec{\tau}_t)) < \rho(t)$, then $\sigma_t \in \gl(\cW_r)$ for all $r\in (d_{\rm{slice}}(\sigma_t, \gl(\vec{\tau}_t))+1,\rho(t)+1)$ by \Cref{P:rgluesurjective}.\endnote{We prove $\sigma_t\in \gl(\cW_r)$ by showing that $d_{\rm{slice}}(\sigma_t, \gl(\vec{\tau}_t))<r-1$, where we note that $\vec{\tau}_t$ is $r$-gluable for all $r<\rho(t)+1$. By hypothesis, $r>d_{\rm{slice}}(\sigma_t, \gl(\vec{\tau}_t))+1$ so that $d_{\rm{slice}}(\sigma_t, \gl(\vec{\tau}_t))<r-1$.} Moreover, it follows from \Cref{L:Limitgluing} that $\rho(t)$ mono\-tonically increases to $\infty$ as $t\to\infty$.

\begin{lem}
\label{L:eventuallyglued}
In \Cref{S:eventuallyglued}, suppose there exists $t_0\in \bb{R}$ such that
\[
    \sup\{\lvert \phi_{\sigma_t}^\pm(F) - \phi_{\tau_{i,t}}(F) \rvert: i \in \{1,\ldots,n\}, F\in \lss \cap \cC_i\} < \rho(t)
\]
$\forall t\ge t_0$. Then $\sigma_t \in \gl(\cG^s)$ for all $t\ge t_0$.
\end{lem}

\begin{proof}
By the discussion after \Cref{D:rho} and the fact that $\cW_r \subset \cG^s$ for all $r$, it suffices to show that $d_{\rm{slice}}(\gl(\vec{\tau}_t),\sigma_t) < \rho(t)$ for all $t\geq t_0$.

By \Cref{L:stronglygluedobjects}, any $\gl(\vec{\tau}_t)$-semistable object of phase $\psi$ has the form $G = G_1\oplus \cdots \oplus G_n$ for some $G_i\in \cQ_{E_i}(\psi - \phi_t(E_i)) \subset \lss \cap \cC_i$. Also, by definition, $\psi = \phi_{E_i}(G_i) + \phi_t(E_i) = \phi_{\tau_i,t}(G_i)$ for all $i$. 

Next, there exists $i$ such that $\phi_{\sigma_t}^+(G_i)=\phi_{\sigma_t}^+(G)$, so $\lvert \phi_{\sigma_t}^+(G) - \psi\rvert = \lvert \phi_{\sigma_t}^+(G_i) - \phi_{\tau_i,t}(G_i)\rvert <\rho(t)$. Similarly, $\lvert \phi^-_{\sigma_t}(G) - \psi\rvert <\rho(t)$, and therefore $d_{\rm{slice}}(\gl(\vec{\tau}_t),\sigma_t)< \rho(t)$ for all $t\ge t_0$. 
\end{proof}

\begin{prop}
In \Cref{S:eventuallyglued}, if $\sigma_\bullet$ satisfies the support property (\Cref{D:supportpropertypath}) and \begin{equation}
\label{E:uniform_bound} \limsup_{t \to \infty} \left( \sup_{F \in \lss} \left( \phi^+_{\sigma_t}(F) - \phi^-_{\sigma_t}(F) \right) \right) < 1.
\end{equation}
then $\sigma_t \in \gl(\cG^s)$ for all $t \gg 0$.
\end{prop}

\begin{proof}
By \Cref{L:eventuallyglued}, it suffices to show that for each $i$, $|\phi^\pm_{\sigma_t}(F) - \phi_{\tau_i,t}(F)| = |\phi^\pm_{\sigma_t}(F) -\phi_{\sigma_t}(E_i)-\phi_{E_i}(F)|$ has an upper bound that is uniform over all $F \in \lss \cap \cC_i$ and all $t$ sufficiently large. The hypothesis \eqref{E:uniform_bound} implies the hypothesis of \Cref{L:logZcomp} for $t\gg 0$, so there is a unique $\theta_t(F) \in [\phi^-_{\sigma_t}(F), \phi^+_{\sigma_t}(F)]$ such that $Z_{\sigma_t}(F) \in \bR_{>0} e^{i \pi \theta_t(F)}$, and $\lim_{t \to \infty} (\phi_{\sigma_t}(F)-\theta_t(F)) = 0$. The triangle inequality combined with \eqref{E:uniform_bound} then shows that it suffices to find a uniform upper bound on $|\theta_t(F)-\theta_{t}(E_i)-\phi_{E_i}(F)|$.\endnote{Specifically, the triangle inequality gives \begin{gather*}|\phi^\pm_{\sigma_t}(F) -\phi_{\sigma_t}(E_i)-\phi_{E_i}(F)| \\ \leq |\phi^\pm_{\sigma_t}(F)-\theta_{t}(F)| + |\theta_t(F) - \theta_t(E_i) - \phi_{E_i}(F)| + |\phi_{\sigma_t}(E_i)-\theta_t(E_i)| \\ \leq |\phi^+_{\sigma_t}(F)-\phi^-_{\sigma_t}(F)| + |\theta_t(F) - \theta_t(E_i) - \phi_{E_i}(F)| + |\phi_{\sigma_t}(E_i)-\theta_t(E_i)|,\end{gather*}where the second inequality uses the fact that $\phi_{\sigma_t}^-(F) \leq \theta_t(F) \leq \phi_{\sigma_t}^+(F)$. The third term in the final expression is independent of $F$ and converges to zero, hence has an upper bound that is uniform over $F$. The first term in that expression is precisely what is bounded above by \eqref{E:uniform_bound}.}

Note that $e(-) := \exp(i\pi (-)) : \bR \to S^1$ is a covering map, and for each individual $F \in \lss \cap \cC_i$, $\theta_t(F)-\theta_{t}(E_i)-\phi_{E_i}(F)$ is a continuous function of $t$ that converges to $0$ as $t \to \infty$. To show that this convergence is uniform over $F$, it suffices to show that the convergence $e(\theta_t(F)-\theta_{t}(E_i)-\phi_{E_i}(F)) \to 1$ is uniform over $F$.

The quantity $e(\theta_t(F)-\theta_{t}(E_i)-\phi_{E_i}(F))$ is the normalization of $Z_{\sigma_t}(F) / (Z_{\sigma_t}(E_i) Z_{E_i}(F))$, so it suffices to show that \[\lim_{t \to \infty} \sup_{F \in \lss \cap \cC_i} \left| \frac {Z_{\sigma_t}(F)}{Z_{\sigma_t}(E_i) Z_{E_i}(F)} - 1 \right| = 0.\]Using the support property, we can bound the quantity as follows:\begin{equation}\label{E:support_property_bound}\left| \frac {Z_{\sigma_t}(F)}{Z_{\sigma_t}(E_i) Z_{E_i}(F)} - 1 \right| \leq \frac{1}{\epsilon_{E_i}} \left| \frac {Z_{\sigma_t}(\widehat{v(F)})}{Z_{\sigma_t}(E_i)} - Z_{E_i}(\widehat{v(F)}) \right|,\end{equation}where $\widehat{v(F)} := v(F) / \lVert v(F) \rVert$ is the normalized Mukai vector of $F$, and $\epsilon_{E_i}>0$ is the constant appearing in \Cref{D:supportpropertypath}. By definition, $Z_{E_i}(-) = \lim_{t \to \infty} Z_{\sigma_t}(-) / Z_{\sigma_t}(E_i)$ in the finite dim\-ensional real vector space $\Hom(\Lambda_i,\bC)$, so $Z_{\sigma_t}(x)/Z_{\sigma_t}(E_i)$ converges to $Z_{E_i}(x)$ uniformly for $x$ in the unit sphere in $\Lambda_i \otimes \bR$. Therefore, the right-hand-side of \eqref{E:support_property_bound} converges to $0$ uniformly over $F \in \cP \cap \cC_i$.
\end{proof}

\subsubsection{Quasi-convergent paths recover semiorthogonal decompositions} 

The following theorem shows that subject to conditions often satisfied in practice, any polarizable semiorthogonal decomposition can be recovered from a numerical quasi-convergent path.

\begin{thm}
\label{T:recoveringsod}
Let $\cC$ be a smooth and proper idempotent complete pre-triangulated dg-category with an semiorthogonal decomposition $\cC = \langle \cC_1,\ldots, \cC_n\rangle$ such that $\Lambda = \bigoplus_i v(\cC_i)$, and let $\vec{\sigma} \in \prod_i \Stab(\cC_i)$. Consider $z_i :[0,\infty)\to \bC$ for $i=1,\ldots,n$ such that for all $i<j$,\endnote{For instance, one can use $z_s(t)=i s t$ for $s=1,\ldots,n$.} 
\[ 
\lim_{t\to \infty} \frac{z_j(t)-z_i(t)}{ 1+|z_j(t)-z_i(t)|} = e^{i \theta_{ij}} \text{ for some } \theta_{ij} \in (0,\pi).
\]
\begin{enumerate} 
    \item $\vec{\sigma}_t := (z_i(t)\cdot\sigma_i)_{i=1}^n \in \prod_i \Stab(\cC_i)$ is strongly gluable $\forall t\gg0$.\vspace{2mm}

    \item The resulting path $\sigma_\bullet := \gl(\vec{\sigma}_\bullet)$ in $\Stab(\cC)$ is numerical and quasi-convergent. \vspace{2mm}

    \item \Cref{P:ImSOD} recovers $\cC=\langle \cC_1,\ldots, \cC_n\rangle$, and \Cref{T:prestabilityonquotient} recovers the stability cond\-itions $\sigma_i$ up to the action of $\bC$.
\end{enumerate}
\end{thm}

\begin{proof}
(1) is by the same argument as the proof of \Cref{L:Limitgluing} and uses \Cref{P:hearthombound}. We next prove that $\sigma_\bullet$ is quasi-convergent.

\medskip
\emph{Characterization of $\lss_{\sigma_\bullet} = \lss$} : Let $\cQ_i$ be the slicing of $\sigma_i$ and $\cQ_{i,t}$ that of $z_i(t)\cdot \sigma_i$. Suppose $t$ is large enough that $\sigma_t=(Z_t,\cQ_t)$ is defined and $\cQ_t(\phi) = \bigoplus_{i=1}^n \cQ_{i,t}(\phi)$ by \Cref{L:stronglygluedobjects}. It follows that $\lss = \bigcup_{i=1}^n \iota_i(\cQ_i)$.

\medskip
\emph{Limit HN filtrations for $\sigma_\bullet$} : Every $X\in \cC$ has a filtration $0 = X_{n+1}\to X_n\to \cdots \to X_1 = X$ with $\Cone(X_{i+1}\to X_i) = G_i\in \cC_i$. $\sigma_{i,t} = z_i(t)\cdot \sigma_i$, so $\sigma_{i,t}$-HN filtrations of $G_i\in \cC_i$ are constant in $t$. Because $\phi_{\sigma_{i,t}}^-(G_i)>\phi_{\sigma_{i-1,t}}^+(G_{i-1})$ for all $i$ for all $t\gg0$, the concatenated filtration is the limit HN filtration by \Cref{P:tstab}.

\medskip
\emph{Verifying} \Cref{D:exit_sequence} \ref{I:semistable_difference} : Take $E\in \iota_i(\cQ_i)$ and $F\in \iota_j(\cQ_j)$ for $i\le j$. Let $L = \lim_t \ell_t(F/E)$. Then, one has 
\begin{align*}
\lim_{t\to\infty}\frac{\ell_t(F/E)}{1+\lvert \ell_t(F/E)\rvert} 
& = 
\begin{cases}
    e^{i\theta_{ij}} & i<j\\
    \frac{L}{1+\lvert L\rvert}& i = j.
\end{cases}
\end{align*}

So, $\sigma_\bullet$ is quasi-convergent. Choose $F_i\in \cQ_i$ for each $i$. It is not hard to see that $\{F_1,\ldots, F_n\} \to \lss/{\sim}$ is a bijection and that $\sim$ and $\isim$ are equivalent.\endnote{$\lss = \bigcup_{i=1}^n \iota_i(\cQ_{i})$ and thus it suffices to observe that $F \in \iota_i(\cQ_i)$ satisfies $F\sim F_i$.} Therefore, for each $1\le i \le n$, $\cC^{F_i} = \cC_i$ and we recover $\cC = \langle \cC_1,\ldots, \cC_n\rangle$ by applying \Cref{P:ImSOD}. $\sigma_\bullet$ is automatically numerical by the hypotheses at the beginning of \S 3.2, so each $\sigma_{F_i}$ factors through $v(\cC_i) = \Lambda_i$. By the definition of $\sigma_{F_i} = (Z_{F_i},\cP_{F_i})$, one sees that $\sigma_{F_i} = Z_i(F_i)^{-1}\cdot \sigma_i$. Therefore, each $\sigma_{F_i}$ satisfies the support property, as it is preserved by $\bC$-action.
\end{proof}

\section{The case of curves}
\label{S:curves}

In this section, we consider $\Stab(X)$ for $X$ a smooth and proper curve over a field $k$ --- see \Cref{E:numerical}. For $g=1$, and $g\ge 2$, $\Stab(X)$ is a torsor for the natural $\GL_2^+(\bR)^\sim$-action by \cite[Thm. 9.1]{Br07} and \cite[Thm. 2.7]{Macricurves}, respectively. This implies $\Stab(X) \cong \bC\times \bH$ as complex manifolds for $g(X) \ge 1$. The case of $\bP^1$ is more complicated; \cite{Ok06} shows that $\Stab(\bP^1) \cong \bC^2$ as complex manifolds. However, an explicit biholomorphism was elusive and only given recently in \cite{NMMP}. 

We will consider paths in $\Stab(X)/\bC$. The map $\Stab(X) \to \Stab(X)/\bC$ is a topologically trivial principal $\bC$-bundle by \cite{deformedmass}*{Cor. 3.5}. We say $\sigma_\bullet:[a,\infty)\to \Stab(X)/\bC$ is \emph{quasi-convergent} if some (equivalently any) lift of it to $\Stab(X)$ is quasi-convergent.\endnote{Given a pair of such lifts, $\tau_\bullet$ and $\eta_\bullet$, there is an associated path $\gamma: [a,\infty) \to \bC$ such that $\gamma(t)\cdot \tau_t = \eta_t$. It follows that $E$ is limit semistable with respect to $\tau_t$ iff it is limit semistable with respect to $\eta_t$. The quantities $\ell_\sigma(E/F)$ are $\bC$-invariant so all of the orders $\preceq, \ipreceq$, etc. are equivalent for $\tau_t$ and $\eta_t$. 
} The $\bC$-invariance of $\ell_\sigma(E/F)$ implies that all of the definitions from \S2 carry over to $\sigma_\bullet$. In particular, a filtration $\{\cC_{\preceq E}\}_{E\in \lss/{\sim}}$ can be attached to $\sigma_\bullet$ by choosing any lift and applying \Cref{T:prestabilityonquotient}.

One argument for considering paths in $\Stab(\cC)/\bC$ rather than in $\Stab(\cC)$ is that in \Cref{T:prestabilityonquotient} we only get stability conditions that are independent of choices of objects modulo $\bC$.

\subsection{The case of \texorpdfstring{$\bP^1$}{P1}}

We recall some relevant parts of the description of $\Stab(\bP^1)/\bC$ from \cite{Ok06}. $\Stab(\bP^1)/\bC$ has an open cover $\{X_k\}_{k\in \bZ}$, where $X_k$ consists of those equivalence classes of stability conditions with respect to which $\cO(k-1)$ and $\cO(k)$ are stable. There is a biholomorphism $\varphi_k:X_k\to \bH$ given by $\sigma \mapsto \logZ_\sigma(\cO(k)) - \logZ_\sigma(\cO(k-1))$. For each pair of $j\ne k$, $X_j \cap X_k = \varphi_k^{-1}(\{z\in \bH: \Im(z)<\pi\})$ is the image in $\Stab(\bP^1)/\bC$ of the set of stability conditions with heart $\Coh(\bP^1)$. Such $\sigma$ are \emph{geometric}, meaning that structure sheaves of points are all $\sigma$-stable of the same phase.

Semiorthogonal decompositions of $\bP^1$ are classified; they all come from full exceptional collections of the form $\langle \cO(k-1),\cO(k)\rangle$ for $k\in \bZ$.\endnote{We include a proof of this for completeness. Suppose given a (nontrivial) semiorthogonal decomposition $\DCoh(\bP^1) = \langle \cA,\cB\rangle$. Both $\cA$ and $\cB$ are closed under summands and in particular $\cA \cap \Coh(\bP^1)\ne 0$ is nontrivial and similarly for $\cB$. Consequently, $\cA$ must contain a coherent sheaf. All of the coherent sheaves on $\bP^1$ are of the form $\cE \oplus \cF$, where $\cE$ is locally free and $\cF$ is a torsion sheaf. However, by the classification of vector bundles on $\bP^1$ and closure of $\cA$ under summands, this implies that $\cA$ contains a torsion sheaf or one of the line bundles $\cO(n)$. 

However, since $\bZ^2 \cong K_0(\bP^1) = K_0(\cA) \oplus K_0(\cB)$ and because $\DCoh(\bP^1)$ contains no phantoms this implies that $\cA$ and $\cB$ must have $K_0 \cong \bZ$. So, $\cA$ cannot contain a torsion sheaf and a line bundle or two distinct line bundles. Therefore, $\cA$ is the extension closure of a line bundle or a torsion sheaf and similarly for $\cB$. If one of $\cA$ or $\cB$ is generated by a torsion sheaf, then Serre duality implies that the needed semiorthogonality conditions do not hold. If $\cA$ and $\cB$ are generated by line bundles $\cO(j)$ and $\cO(k)$, respectively, then the condition $\RHom(\cO(k),\cO(j)) = 0$ is equivalent to $k = j+1$.}

\begin{prop}
    The stability conditions in $\Stab(\bP^1)/\bC$ glued from $\langle \cO(k-1),\cO(k)\rangle$ are exactly $\varphi_k^{-1}(\{z\in \bH: \Im(z)>\pi\})$.
\end{prop}

\begin{proof}
    We choose representatives for classes in $X_k$ in $\Stab(\bP^1)/\bC$ by setting $\logZ_\sigma(\cO(k-1)) = 1$. If $\Im(\varphi_k(\sigma)) > \pi$, then by \cite[Prop. 3.3]{Ok06} $\cP_\sigma(0,1] = \left[\cO(k-1)[1],\cO(k)[q]\right]$ for the unique integer $q$ such that $1-q\in ( \phi_\sigma(\cO),\phi_\sigma(\cO) + 2).$ By \Cref{D:generalglued}, such $\sigma$ are glued from stability conditions on $\langle \cO(k-1)\rangle$, $\langle \cO(k)\rangle$. It is also true that all $\sigma\in\varphi_k^{-1}(\{z\in \bH: \Im(z) \in (0,\pi)\})$ are not glued, but we omit this since it is not used in what follows.\endnote{The geometric stability conditions cannot be glued from $\langle \cO(k-1),\cO(k)\rangle$ for any $k$. Indeed, the hearts of geometric stability conditions on $\bP^1$ contain $\cO(n)$ for all $n$. However, a suitable glued heart would be of the form $[\cO(k-1),\cO(k)]$, which does not contain $\cO(n)$ for $n\not\in\{k-1,k\}$.}
\end{proof}
$\Pic(\bP^1)\subset \Aut(\DCoh(\bP^1))$ acts freely on $\Stab(\bP^1)$ and the action descends to $\Stab(\bP^1)/\bC$. For any $k$, $\Pic(\bP^1) \cdot X_k = \Stab(\bP^1)/\bC$ and $\cO(1)\cdot X_k\subseteq X_{k+1}$. So, $X_k$ contains a fundamental domain $\Omega_k$ for the $\Pic(\bP^1)$-action. In the $\varphi_k$ coordinate, $\Omega_k = \varphi_k^{-1}\{(x,y)\in \bH: y>0, \cos y \ge e^{-\lvert x\rvert}\}$ (see \cite[Lem. 4.3]{Ok06} and the figure following it). It follows that $\Omega_k$ contains the entire glued region corresponding to $\langle \cO(k-1),\cO(k)\rangle$ and some of the geometric stability conditions. 

In \cite[Prop. 26]{NMMP}, an explicit biholomorphism $\mathscr{B}:\bC\to \Stab(\bP^1)/\bC$ is given using solutions to the quantum differential equation. $\mathscr{B}$ is constructed by gluing maps $\mathscr{B}_k: \bR + i\pi[k-1,k] \to \Stab(\bP^1)/\bC$ where $\mathscr{B}_k$ maps biholomorphically (on the interior of its domain) to $\Omega_k$. The action of $\cO(1)$ is identified with $\tau \mapsto \tau +i\pi$ on $\bC$.

Next, we consider the paths arising as solutions of the quantum differential equation in \cite{NMMP}, parametrized using $\mathscr{B}$. Given an initial point $z_0 \in \bC$ with $\Im(z_0) \in \pi [k-1,k]$ and $\mathscr{B}_k(z_0) \in \Omega_k$, the path from the quantum differential equation is $\mathscr{B}_k(z_0 + \ln t)$ for $t > 0$, which stays in $\Omega_k$. By loc. cit., the resulting path in the $\varphi_k$ coordinate is:
\begin{equation}
\label{E:asymptoticestimate}
\varphi_k(t) = 2\kappa t + i\cdot \frac{\pi}{2} + O(\lvert \kappa t\rvert^{-1})
\end{equation}
where $\kappa$ lies on the ray $\bR_{>0}\cdot e^{\kappa - (k-1)\pi i}$. $\varphi_k(t) = \logZ_t(\cO(k)) - \logZ_t(\cO(k-1))$ so \eqref{E:asymptoticestimate} combined with \Cref{L:logZcomp} gives 
\[
\lim_{t\to\infty} \frac{\ell_t(\cO(k)/\cO(k-1))}{1+ \lvert \ell_t(\cO(k)/\cO(k-1))\rvert} = \lim_{t\to\infty} \frac{2\kappa t + i\pi/2}{1+ \lvert 2\kappa t + i\pi/2 \rvert} = \frac{\kappa}{\lvert \kappa \rvert}.
\]

If $\Im(\kappa) > 0$, then $\mathscr{B}_k(z_0+\ln(t))$ enters the glued region of $X_k$ as $t\to\infty$. Hence, $\lss$ consists of $\cO(k-1)^{\oplus r}$, $\cO(k)^{\oplus s}$ for all $r,s\ge 1$ and all of their shifts. 
$\{\cO(k-1),\cO(k)\} \to \lss/{\sim}$ is a bijection and $\cO(k-1)\iprec \cO(k)$. \Cref{P:ImSOD} recovers the semiorthogonal decomposition $\langle \cO(k-1),\cO(k)\rangle$. 

If $\Im(\kappa) = 0$ and $\Re(\kappa) > 0$ then $\varphi_k(t)$ is in the geometric region of $X_k$ for all $t\gg0$ by \eqref{E:asymptoticestimate}. Therefore, $\lss$ consists of sums and shifts of structure sheaves of points together with $\{\cO(n)^{\oplus r}: n\in \bZ,r\ge 1\}$. A quick calculation verifies that all of the relevant limits in \Cref{D:exit_sequence} \ref{I:semistable_difference} exist. Representatives for $\lss/{\sim}$ are again given by $\{\cO(k-1),\cO(k)\}$, except now $\cO(k-1) \isim \cO(k)$ and $\cO(k-1)\prec \cO(k)$. We arrive at a two step filtration
\[
0\subsetneq \langle \cO(k-1)\rangle \subsetneq \DCoh(\bP^1).
\]
If $\Re(\kappa) < 0$ then $\varphi_k(t)$ is in the geometric region for all $t\gg0$. $\{\cO(k-1),\cO(k)\}\to \lss/{\sim}$ is again a bijection, and $\cO(k-1) \isim \cO(k)$, but now $\cO(k)\prec \cO(k-1)$ and we arrive at a different two step filtration
\[
0\subsetneq \langle \cO(k)\rangle \subsetneq \DCoh(\bP^1).
\]
Here, all of the associated graded categories are equivalent to $\DCoh(\pt)$.

\subsection{The case of \texorpdfstring{$g(X)\ge 1$}{g(x)>=1}}

If $g(X) \ge 1$, there are no glued stability conditions, since $\DCoh(X)$ is indecomposable \cite{OkawaSODs}. Consequently, in contrast to the $g(X) = 0$ case considered previously, non-trivial filtrations of $\DCoh(X)$ where $g(X) \ge 1$ are never admissible. 

As mentioned above, $\Stab(X) \cong \bC \times \bH$. It follows from \cite[Thm. 2.7]{Macricurves} that the stable objects of any $\sigma \in \Stab(X)$ are precisely the $\mu$-stable vector bundles and the point sheaves. $[\sigma] \mapsto Z_\sigma(\cO_p)/Z_\sigma(\cO_X)$ induces a biholomorphism $\Stab(X)/\bC \cong \bH$. Hence, a path $Z_t$ in the space of projectivized central charges, $\bP\Hom(\Lambda,\bC)$, lifts to $\Stab(X)/\bC$ if and only if $Z_t(\cO_p)/Z_t(\cO_X)\in \bH$ for all $t$.

In the case of $g(X) = 1$, $K_X\cong \cO_X$ and the quantum differential equation is trivial. Therefore, there are no non-trivial paths arising from the quantum differential equation.

In the $g(X)\ge2$ case, the canonical fundamental solution of the quantum differential equation does \emph{not} lift to a path convergent in $\Stab(X)/\bC$ \cite[p. 28]{NMMP}. Denote the coordinate on $\bH$ by $\tau$. The associated path is 
\[
\tau(t) = \frac{2\pi i}{e^{i\theta}t + 2(g-1)C_{\rm{eu}}}
\]
where $C_{\rm{eu}}$ is the Euler--Mascheroni constant, and choosing $-\pi/2 < \theta < \pi/2$ ensures that the path lifts for all $t$. In particular, $Z_t(\cO_p)/Z_t(\cO_X) = \tau(t)$, and $\lim_{t\to\infty} Z_t(\cO_p)/Z_t(\cO_X) = 0$. Let $\sigma_\bullet$ denote the resulting path in $\Stab(X)/\bC$. 

In what follows, $\cT \subset \DCoh(X)$ denotes the full subcategory whose objects are the torsion complexes, i.e. those complexes set-theoretically supported in finitely many points of $X$. Also, $K(X)$ denotes the field of rational functions on $X$.

\begin{lem}
The path $\sigma_\bullet$ in $\Stab(X)/\bC$ is quasi-convergent and
    \begin{enumerate} 
        \item $\lss_{\sigma_\bullet}$ consists of the $\mu$-semistable coherent sheaves on $X$ and their shifts; \vspace{2mm}
        \item the filtration of $\DCoh(X)$ from $\sigma_\bullet$ is $0\subsetneq  \cT \subsetneq \DCoh(X)$ and $\DCoh(X)/\cT \simeq \DCoh(K(X))$; \vspace{2mm}
        \item the filtration of $H^*_{\rm{alg}}(X)$ is $0\subsetneq H^0(X;\bZ) \subsetneq H^*_{\rm{alg}}(X)$; and \vspace{2mm}
        \item choosing any pair of objects $\cE,\cF\in \lss$ such that $\{\cE,\cF\}\to \lss/{\sim}$ is a bijection, the induced pre-stability conditions on the associated graded categories satisfy the support property.
    \end{enumerate}
\end{lem}

\begin{proof}
Existence of limit HN filtrations for $\sigma_\bullet$ follows from existence of $\sigma_t$-HN filtrations for each $t\ge 0$, which are furthermore constant in $t$. For any central charge $Z$ factoring through $\ch:K_0(X)\to H^*_{\rm{alg}}(X)$, one has $Z(\cE) = \rm{rk}(\cE)\cdot Z(\cO_X) + \deg(\cE)\cdot Z(\cO_p)$. We abbreviate this by $Z(r,d)$ where $r = \rm{rk}(\cE)$ and $d = \deg(\cE)$. By \Cref{L:logZcomp}, it suffices to show that $\logZ_t(r_1,d_1) - \logZ_t(r_2,d_2)$ satisfies \Cref{D:exit_sequence} \ref{I:semistable_difference}. I.e., we analyze 
\begin{equation}
\label{E:logdifferencecalc}
\logZ_t(r_1,d_1) - \logZ_t(r_2,d_2) = \log\left(\frac{r_1+ d_1 \tau(t)}{r_2 +d_2\tau(t)}\right).
\end{equation}
If both $r_1$ and $r_2$ are nonzero, then the limit of \eqref{E:logdifferencecalc} exists in $\bC$. If $r_1 = 0$ and $r_2 = 0$, then $d_1$ and $d_2$ are both nonzero and the limit again exists. If $r_1 = 0$ and $r_2\ne 0$, then \eqref{E:logdifferencecalc} equals $\log(2\pi i d_1) - \log(t) - i\theta$. If $r_1 \ne 0$ and $r_2 = 0$, \eqref{E:logdifferencecalc} equals $\log(t) - \log(2\pi i d_2) + i\theta$. This verifies \Cref{D:exit_sequence} \ref{I:semistable_difference}. Also, this suggests a natural set of representatives for $\lss/{\sim}$, namely $\{\cO_p,\cO_X\}$. $\cO_p\prec \cO_X$, however $\cO_p\isim \cO_X$. Hence, $\sigma_\bullet$ gives rise to a two step filtration by real asymptotics by \Cref{P:exit_sequence_filtration}. For $\cE\in \lss$, $\cE\sim \cO_p$ if and only if $\cE$ is a shift of a torsion sheaf and hence $\DCoh(X)_{\preceq \cO_p} = \cT$. The claimed filtration of $\DCoh(X)$ follows. $\DCoh(X)/\cT \simeq \DCoh(K(X))$ by \cite{Meinhardtquotient}*{Prop. 3.13}. 

$K_0(\cT)$ is infinite rank, however $K_0(\cT) \to K_0(X) \to H^*_{\rm{alg}}(X)$ is given by sending $\cF \in \cT$ to $\deg(\cF)$. So, we obtain a filtration $0\subsetneq H^0(X;\bZ) \subsetneq H^*_{\rm{alg}}(X)$ induced by the filtration of $\DCoh(X)$. Since both $H^0(X;\bZ)$ and $H^*_{\rm{alg}}(X)/H^0(X;\bZ) \cong H^2(X;\bZ)$ are free of rank $1$, any pre-stability conditions on the associated graded subcategories induced by $\sigma_\bullet$ factors through a rank $1$ lattice and thus satisfy the support property.\endnote{Choose a norm on $H^0(X;\bZ)\otimes \bR \cong \bR$. The quantity relevant to the support property is $\inf_{E\in \cP}\lvert Z(E)\rvert/\lVert v(E)\rVert$, but this quantity is invariant under scale. So, $\inf_{E\in \cP}\lvert Z(E)\rvert/\lVert v(E)\rVert = \lvert Z(F)\rvert/\lVert v(F)\rVert>0$, where $F$ is any semistable object.}
\end{proof}

\printendnotes

\bibliography{refs}{}

\begin{bibdiv}
\begin{biblist}

\bib{Bayer_short}{article}{
      author={Bayer, Arend},
       title={A short proof of the deformation property of {B}ridgeland stability conditions},
        date={2019},
     journal={Mathematische Annalen},
      volume={375},
      number={3},
       pages={1597\ndash 1613},
         url={https://doi.org/10.1007/s00208-019-01900-w},
}

\bib{BvdB02}{article}{
      author={Bondal, Alexey},
      author={Bergh, Michel},
       title={Generators and representability of functors in commutative and noncommutative geometry},
        date={200205},
     journal={Moscow Mathematical Journal, v.3, 1-36 (2003)},
      volume={3},
}

\bib{Beilinson1978}{article}{
      author={Beilinson, Alexander~A.},
       title={Coherent sheaves on {$\mathbb{P}^n$} and problems of linear algebra},
        date={1978Jul},
        ISSN={1573-8485},
     journal={Functional Analysis and Its Applications},
      volume={12},
      number={3},
       pages={214\ndash 216},
         url={https://doi.org/10.1007/BF01681436},
}

\bib{BlancTopK}{article}{
      author={Blanc, Anthony},
       title={Topological {K}-theory of complex noncommutative spaces},
        date={2016},
     journal={Compositio Mathematica},
      volume={152},
      number={3},
       pages={489–555},
}

\bib{Br07}{article}{
      author={Bridgeland, Tom},
       title={Stability conditions on triangulated categories},
        date={2007},
        ISSN={0003-486X},
     journal={Annals of Mathematics. Second Series},
      volume={166},
      number={2},
       pages={317\ndash 345},
         url={https://doi.org/10.4007/annals.2007.166.317},
      review={\MR{2373143}},
}

\bib{mmp}{article}{
      author={Ballard, Matthew},
      author={Diemer, Colin},
      author={Favero, David},
      author={Katzarkov, Ludmil},
      author={Kerr, Gabriel},
       title={The {M}ori program and non-{F}ano toric homological mirror symmetry},
        date={2015},
        ISSN={0002-9947,1088-6850},
     journal={Trans. Amer. Math. Soc.},
      volume={367},
      number={12},
       pages={8933\ndash 8974},
         url={https://doi.org/10.1090/S0002-9947-2015-06541-6},
      review={\MR{3403076}},
}

\bib{BGvBSJH}{article}{
      author={Böhning, Christian},
      author={{Graf von Bothmer}, Hans-Christian},
      author={Sosna, Pawel},
       title={On the {J}ordan–{H}ölder property for geometric derived categories},
        date={2014},
        ISSN={0001-8708},
     journal={Advances in Mathematics},
      volume={256},
       pages={479\ndash 492},
         url={https://www.sciencedirect.com/science/article/pii/S0001870814000668},
}

\bib{BGvBKSPhantom}{article}{
      author={Böhning, Christian},
      author={Graf~v. Bothmer, Hans-Christian},
      author={Katzarkov, Ludmil},
      author={Sosna, Pawel},
       title={Determinantal {B}arlow surfaces and phantom categories},
        date={201210},
     journal={Journal of the European Mathematical Society},
      volume={17},
}

\bib{B-KSerre}{article}{
      author={Bondal, Alexey~I.},
      author={Kapranov, Mikhail~M.},
       title={Representable functors, {S}erre functors, and mutations},
        date={1990jun},
     journal={Mathematics of the USSR-Izvestiya},
      volume={35},
      number={3},
       pages={519},
         url={https://dx.doi.org/10.1070/IM1990v035n03ABEH000716},
}

\bib{BMSAbelian}{article}{
      author={Bayer, Arend},
      author={Macr{\`{\i}}, Emanuele},
      author={Stellari, Paolo},
       title={The space of stability conditions on abelian threefolds, and on some {Calabi}-{Yau} threefolds},
    language={English},
        date={2016},
        ISSN={0020-9910},
     journal={Inventiones Mathematicae},
      volume={206},
      number={3},
       pages={869\ndash 933},
}

\bib{BMT}{article}{
      author={Bayer, Arend},
      author={Macrì, Emanuele},
      author={Toda, Yukinobu},
       title={Bridgeland stability conditions on threefolds {I}: Bogomolov-{G}ieseker type inequalities},
        date={2014},
        ISSN={1056-3911},
     journal={Journal of Algebraic Geometry},
      volume={23},
       pages={117\ndash 163},
}

\bib{BNeeman}{article}{
      author={Bökstedt, Marcel},
      author={Neeman, Amnon},
       title={Homotopy limits in triangulated categories},
    language={eng},
        date={1993},
     journal={Compositio Mathematica},
      volume={86},
      number={2},
       pages={209\ndash 234},
         url={http://eudml.org/doc/90218},
}

\bib{BondalOrlovsod}{misc}{
      author={Bondal, Alexey},
      author={Orlov, Dmitri},
       title={Semiorthogonal decomposition for algebraic varieties},
        date={1995},
}

\bib{CP10}{article}{
      author={Collins, John},
      author={Polishchuk, Alexander},
       title={Gluing stability conditions},
        date={2010},
        ISSN={1095-0761},
     journal={Adv. Theor. Math. Phys.},
      volume={14},
      number={2},
       pages={563\ndash 607},
         url={http://projecteuclid.org.proxy.library.cornell.edu/euclid.atmp/1288619153},
      review={\MR{2721656}},
}

\bib{Drinfelddg}{article}{
      author={Drinfeld, Vladimir},
       title={{D}{G} quotients of {D}{G} categories},
        date={2004},
        ISSN={0021-8693},
     journal={Journal of Algebra},
      volume={272},
      number={2},
       pages={643\ndash 691},
         url={https://www.sciencedirect.com/science/article/pii/S0021869303005829},
}

\bib{symplectomorphism}{article}{
      author={Diemer, Colin},
      author={Katzarkov, Ludmil},
      author={Kerr, Gabriel},
       title={Symplectomorphism group relations and degenerations of {L}andau--{G}inzburg models},
        date={2016},
        ISSN={1435-9855,1435-9863},
     journal={Journal of the European Mathematical Society},
      volume={18},
      number={10},
       pages={2167\ndash 2271},
         url={https://doi.org/10.4171/JEMS/640},
      review={\MR{3551190}},
}

\bib{GKRpaper04}{article}{
      author={Gorodentsev, Alexey~L.},
      author={Kuleshov, S~A},
      author={Rudakov, Alexey~N.},
       title={t-stabilities and t-structures on triangulated categories},
        date={2004aug},
     journal={Izvestiya: Mathematics},
      volume={68},
      number={4},
       pages={749},
         url={https://dx.doi.org/10.1070/IM2004v068n04ABEH000497},
}

\bib{GorchOrlov}{article}{
      author={Gorchinskiy, Sergey},
      author={Orlov, Dmitri},
       title={Geometric phantom categories},
        date={2013-06},
        ISSN={1618-1913},
     journal={Publications Mathématiques de l’IHÉS},
      volume={117},
       pages={329–349},
         url={http://dx.doi.org/10.1007/s10240-013-0050-5},
}

\bib{NMMP}{misc}{
      author={Halpern-Leistner, Daniel},
       title={The noncommutative minimal model program},
        date={2023},
}

\bib{HL_robotis}{misc}{
      author={Halpern-Leistner, Daniel},
      author={Robotis, Antonios-Alexandros},
       title={The space of augmented stability conditions},
        date={2025},
         url={https://arxiv.org/abs/2501.00710},
}

\bib{deformedmass}{misc}{
      author={Halpern-Leistner, Daniel},
      author={Robotis, Antonios-Alexandros},
       title={Properties of deformed mass and phase functions},
        date={2026},
         url={https://arxiv.org/abs/2606.00396},
}

\bib{Kawamata2008derived}{incollection}{
      author={Kawamata, Yujiro},
       title={Birational geometry and derived categories},
        date={2018},
   booktitle={Surveys in differential geometry},
      volume={22},
   publisher={International Press},
       pages={291\ndash 317},
}

\bib{Kellerdgcat}{incollection}{
      author={Keller, Bernhard},
       title={On differential graded categories},
        date={2006},
   booktitle={Proceedings of the international congress of mathematicians},
      volume={II},
   publisher={European Mathematical Society},
     address={Zürich},
       pages={151\ndash 190},
}

\bib{KrahPhantom}{article}{
      author={Krah, Johannes},
       title={A phantom on a rational surface},
        date={2024},
     journal={Inventiones Mathematicae},
      volume={235},
      number={3},
       pages={1009\ndash 1018},
}

\bib{K09}{book}{
      author={Krause, Henning},
      editor={Holm, Thorsten},
      editor={Jørgensen, Peter},
      editor={Rouquier, Raphaël},
       title={Localization theory for triangulated categories},
      series={London Mathematical Society Lecture Note Series},
   publisher={Cambridge University Press},
        date={2010},
}

\bib{KS08}{misc}{
      author={Kontsevich, Maxim},
      author={Soibelman, Yan},
       title={Stability structures, motivic {D}onaldson--{T}homas invariants and cluster transformations},
   publisher={arXiv},
        date={2008},
         url={https://arxiv.org/abs/0811.2435},
}

\bib{Macricurves}{article}{
      author={Macrì, Emanuele},
       title={Stability conditions on curves},
        date={200706},
     journal={Mathematical Research Letters},
      volume={14},
}

\bib{Meinhardtquotient}{article}{
      author={Meinhardt, Sven},
      author={Partsch, Holger},
       title={Quotient categories, stability conditions, and birational geometry},
        date={2014},
     journal={Geometriae Dedicata},
      volume={173},
      number={1},
       pages={365\ndash 392},
         url={https://doi.org/10.1007/s10711-013-9947-x},
}

\bib{NeemanTriangulated}{book}{
      author={Neeman, Amnon},
       title={Triangulated categories. (am-148), volume 148},
   publisher={Princeton University Press},
     address={Princeton},
        date={2001},
        ISBN={9781400837212},
         url={https://doi.org/10.1515/9781400837212},
}

\bib{Ok06}{article}{
      author={Okada, So},
       title={Stability manifold of {${\Bbb P}^1$}},
        date={2006},
        ISSN={1056-3911},
     journal={J. Algebraic Geom.},
      volume={15},
      number={3},
       pages={487\ndash 505},
         url={https://doi-org.proxy.library.cornell.edu/10.1090/S1056-3911-06-00432-2},
      review={\MR{2219846}},
}

\bib{OkawaSODs}{article}{
      author={Okawa, Shinnosuke},
       title={Semi-orthogonal decomposability of the derived category of a curve},
        date={2011},
        ISSN={0001-8708},
     journal={Advances in Mathematics},
      volume={228},
      number={5},
       pages={2869\ndash 2873},
         url={https://www.sciencedirect.com/science/article/pii/S0001870811002726},
}

\bib{Or16}{article}{
      author={Orlov, Dmitri},
       title={Smooth and proper noncommutative schemes and gluing of {DG} categories},
        date={2016oct},
     journal={Advances in Mathematics},
      volume={302},
       pages={59\ndash 105},
         url={https://doi.org/10.10162Fj.aim.2016.07.014},
}

\bib{Schlichting06}{article}{
      author={Schlichting, Marco},
       title={Negative {K}-theory of derived categories},
        date={2006May},
        ISSN={1432-1823},
     journal={Mathematische Zeitschrift},
      volume={253},
      number={1},
       pages={97\ndash 134},
         url={https://doi.org/10.1007/s00209-005-0889-3},
}

\bib{stacks-project}{misc}{
      author={{Stacks Project Authors}, The},
       title={\textit{Stacks Project}},
         how={\url{https://stacks.math.columbia.edu}},
        date={2018},
}

\bib{DerivedMorita}{article}{
      author={To{\"e}n, Bertrand},
       title={The homotopy theory of dg-categories and derived morita theory},
        date={2007-03},
     journal={Inventiones Mathematicae},
      volume={167},
      number={3},
       pages={615\ndash 667},
}

\bib{ToenVaquie}{article}{
      author={To\"{e}n, Bertrand},
      author={Vaqui\'{e}, Michel},
       title={Moduli of objects in dg-categories},
        date={2007},
        ISSN={0012-9593},
     journal={Ann. Sci. \'{E}cole Norm. Sup. (4)},
      volume={40},
      number={3},
       pages={387\ndash 444},
         url={https://doi.org/10.1016/j.ansens.2007.05.001},
      review={\MR{2493386}},
}

\bib{Verdierquotient}{inproceedings}{
      author={Verdier, Jean-Louis},
       title={Categories derivees quelques r{\'e}sultats (etat 0)},
        date={1977},
   booktitle={Cohomologie etale},
   publisher={Springer Berlin Heidelberg},
     address={Berlin, Heidelberg},
       pages={262\ndash 311},
}

\bib{Woolfmetric}{article}{
      author={Woolf, J.},
       title={{Some metric properties of spaces of stability conditions}},
        date={201207},
        ISSN={0024-6093},
     journal={Bulletin of the London Mathematical Society},
      volume={44},
      number={6},
       pages={1274\ndash 1284},
      eprint={https://academic.oup.com/blms/article-pdf/44/6/1274/1072881/bds056.pdf},
         url={https://doi.org/10.1112/blms/bds056},
}

\end{biblist}
\end{bibdiv}
\bibliographystyle{plain}

\end{document}